\documentclass[11pt]{article}

\usepackage{amsmath,amssymb,amsthm,bm,diagbox,enumerate,ifxetex,mathtools,microtype,subcaption,tikz}
\usepackage{hyperref}
\usepackage[margin=1in]{geometry}
\usepackage[all]{hypcap}

\ifxetex
  \usepackage{fontspec}
\else
  \usepackage[T1]{fontenc}
  \usepackage[utf8]{inputenc}
  \usepackage{lmodern}
\fi

\DeclareMathOperator{\Arith}{Arith}
\DeclareMathOperator{\diag}{diag}
\DeclareMathOperator{\Image}{Im}
\DeclareMathOperator{\SArith}{SArith}
\DeclareMathOperator{\Sub}{Sub}

\DeclarePairedDelimiter{\abs}{\lvert}{\rvert}
\DeclarePairedDelimiter{\ceil}{\lceil}{\rceil}
\DeclarePairedDelimiter{\floor}{\lfloor}{\rfloor}

\newcommand{\overbar}[1]{\mkern 0.6mu\overline{\mkern-0.6mu#1\mkern-0.6mu}\mkern 0.6mu}

\newcommand{\vd}{\mathbf{d}}
\newcommand{\vr}{\mathbf{r}}
\newcommand{\vzero}{\mathbf{0}}
\newcommand{\rab}{\overbar{\mathbf{r}}}
\newcommand{\dd}{d}
\renewcommand{\r}{\overbar{r}}
\newcommand{\rr}{r}
\newcommand{\q}{\mathbf{q}}
\newcommand{\ZZ}{\mathbb{Z}}
\newcommand{\BB}{\mathcal{B}}

\newtheorem{proposition}{Proposition}[section]
\newtheorem{lemma}[proposition]{Lemma}
\newtheorem{corollary}[proposition]{Corollary}
\newtheorem{theorem}[proposition]{Theorem}
\newtheorem*{genericthm*}{\thistheoremname}
\newenvironment{namedthm*}[1]
  {\renewcommand{\thistheoremname}{#1}%
   \begin{genericthm*}}
  {\end{genericthm*}}
\theoremstyle{definition}
\newcommand{\thistheoremname}{}

\title{Arithmetical structures on bidents}
\author{Kassie Archer\footnote{Department of Mathematics, University of Texas at Tyler, 3900 University Blvd., Tyler, TX 75799, USA; Email address: karcher@uttyler.edu}, Abigail C. Bishop\footnote{Department of Mathematics and Physics, Iona College, 715 North Ave, New Rochelle, NY 10801, USA; Email address: abishop@iona.edu}, Alexander Diaz-Lopez\footnote{Department of Mathematics and Statistics, Villanova University, 800 Lancaster Ave (SAC 305), Villanova, PA 19085, USA; Email address: alexander.diaz-lopez@villanova.edu},\\ Luis D. Garc\'ia Puente\footnote{Department of Mathematics and Statistics, Sam Houston State University, Huntsville, TX 77341, USA; Email address: lgarcia@shsu.edu}, Darren Glass\footnote{Department of Mathematics, Gettysburg College, 300 N. Washington St, Gettysburg, PA 17325, USA; Email address: dglass@gettysburg.edu}, Joel Louwsma\footnote{Department of Mathematics, Niagara University, Niagara University, NY 14109, USA; Email address: jlouwsma@niagara.edu}}
\date{}

\begin{document}
\maketitle

\begin{abstract}
An arithmetical structure on a finite, connected graph $G$ is a pair of vectors $(\mathbf{d}, \mathbf{r})$ with positive integer entries for which $(\operatorname{diag}(\mathbf{d}) - A)\mathbf{r} = \mathbf{0}$, where $A$ is the adjacency matrix of $G$ and where the entries of $\mathbf{r}$ have no common factor. The critical group of an arithmetical structure is the torsion part of the cokernel of $(\operatorname{diag}(\mathbf{d}) - A)$. In this paper, we study arithmetical structures and their critical groups on bidents, which are graphs consisting of a path with two ``prongs'' at one end. We give a process for determining the number of arithmetical structures on the bident with $n$ vertices and show that this number grows at the same rate as the Catalan numbers as $n$ increases. We also completely characterize the groups that occur as critical groups of arithmetical structures on bidents.
\end{abstract}

\section{Introduction}

Arithmetical structures on graphs generalize the notion of the Laplacian matrix of a graph; similarly, the associated critical groups generalize the sandpile group of a graph. Arithmetical structures and their critical groups were introduced by Lorenzini~\cite{L89} as intersection matrices and the associated group of components that arise when studying degenerating curves in algebraic geometry. In this paper, we analyze the combinatorics of arithmetical structures and their critical groups on bidents, which we define to be the graphs illustrated in Figure~\ref{fig:bident}.

\begin{figure}
\begin{center}
\begin{tikzpicture}
\node (0) at (-1.5,1) [circle,draw=black,fill=black, label=below:{$v_x$},inner sep=0pt, minimum size=.18cm]{};
\node (1) at (-1.5,-1) [circle,draw=black,fill=black, label=below:{$v_y$},inner sep=0pt, minimum size=.18cm]{};
\node (a) at (0,0) [circle,draw=black,fill=black, label=below:{$v_0$},inner sep=0pt, minimum size=.18cm]{};
\node (b) at (1.75,0) [circle,draw=black,fill=black, label=below:{$v_1$},inner sep=0pt, minimum size=.18cm]{};
\node (c) at (3.5,0) {$\boldsymbol{\cdots}$};
\node (d) at (5.25,0) [circle,draw=black,fill=black, label=below:{$v_{\ell}$},inner sep=0pt, minimum size=.18cm]{};
\draw (0)--(a);
\draw (1)--(a);
\draw (a) -- (b);
\draw (b) -- (c);
\draw (c) -- (d);
\end{tikzpicture}
\end{center}
\caption{The bident \texorpdfstring{$D_n$}{D\_n} on $n=\ell+3$ vertices. The labeling shown here will be used throughout the paper.\label{fig:bident}}
\end{figure}
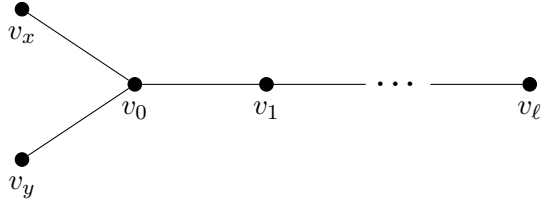

Let $G$ be a finite, connected graph with $n$ vertices, and let $A$ be the adjacency matrix of $G$.
An \emph{arithmetical structure} on $G$ is given by a pair of vectors $(\vd, \vr) \in \ZZ_{> 0}^n\times \ZZ_{> 0}^n$ for which the matrix $L(G,\vd)\coloneqq (\diag(\vd)-A)$ satisfies the equation
\[
L(G, \vd)\vr = \vzero,
\]
and where  the entries of $\vr$ have no nontrivial common factor. The \emph{Laplacian matrix} of $G$, defined to be $L(G)\coloneqq(\diag(\vd')-A)$, where $\vd'$ is the vector with entries given by the degrees of the vertices of $G$, gives an arithmetical structure on $G$. Indeed, taking $\vr'$ to be the vector all of whose entries are $1$, we have that $L(G, \vd')\vr' = \vzero$.

It was shown in~\cite[Proposition 1.1]{L89} that for any arithmetical structure $(\vd,\vr)$ on $G$ the matrix $L(G,\vd)$ has rank $n-1$, implying that the choice of $\vd$ uniquely determines $\vr$ and vice versa.  Moreover, it follows that the associated linear transformation $L(G,\vd)\colon\ZZ^n\to\ZZ^n$ has cokernel $\ZZ^n/\Image(L(G,\vd))$ of the form $\ZZ\oplus K(G;\vd,\vr)$ for some finite abelian group $K(G;\vd,\vr)$. This group $K(G;\vd,\vr)$ is called the \emph{critical group} of the arithmetical structure $(\vd,\vr)$. In terms of arithmetic geometry, this group is isomorphic to the group of components of a N\'eron model associated to the Jacobian of a curve~\cite{L89}.  In the special case of the Laplacian, the critical group is also known as the \emph{sandpile group}, an object that has become a crossroads of a wide range of mathematics, physics, and computer science. For more information, see~\cite{CP}, among others.

Lorenzini~\cite[Lemma 1.6]{L89} also shows that any finite, connected graph has a finite number of arithmetical structures; however, his proof does not give a bound on the number of such structures.  Recent work in \cite{AB,B17,C18,C17} involves studying arithmetical structures and their critical groups on various families of connected graphs. In~\cite{B17}, the authors show that the number of arithmetical structures on the path graph $P_n$ is given by the Catalan number $C(n-1)$ and that the number of arithmetical structures on the cycle graph $C_n$ is given by the binomial coefficient $\binom{2n-1}{n-1}$. For the star $K_{n,1}$, the number of arithmetical structures was shown in~\cite{C18} to be given by the number of positive integer solutions to the Diophantine equation
\[
d_0=\sum_{i=1}^n \frac{1}{d_i}.
\]
These solutions are so-called Egyptian fraction representations of $d_0$ (see~\cite[A280517]{OEIS}). If we impose the condition $d_0=1$, the number of positive integer solutions to the resulting equation is the number of arithmetical structures on the complete graph $K_{n}$.

The results of~\cite{B17} completely classify the arithmetical structures on graphs where all vertices have degree $1$ or $2$, so in this article we consider \emph{bidents}, pictured in Figure~\ref{fig:bident}, as they are the simplest family of graphs that have a vertex of degree $3$. We denote these graphs by $D_n$ since they are isomorphic as graphs to Dynkin diagrams of type $D$, and we define them for all $n\geq3$. We use the indices of the labels in Figure~\ref{fig:bident} to index the vectors $\vd$ and $\vr$ for any given arithmetical structure $(\vd, \vr)$ on $D_n$. In particular, we write $\vd=(d_x, d_y, d_0, d_1, \dotsc, d_\ell)$ and $\vr=(r_x, r_y, r_0, r_1, \dotsc, r_\ell)$ and label graphs with their respective $\vd$- and $\vr$-labelings, as shown in  Figure~\ref{fig:bidentlabels}.

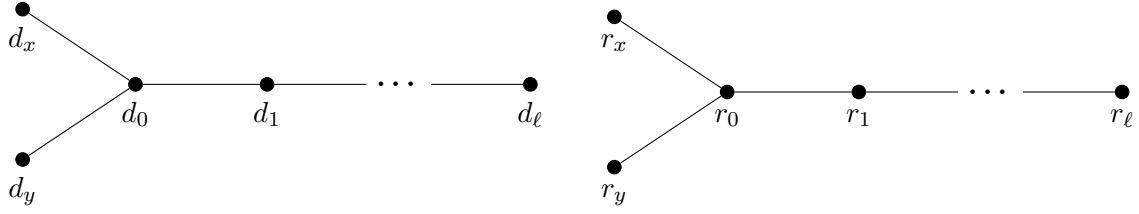
\begin{figure}
\begin{center}
\begin{tikzpicture}
\node (0) at (-1.5,1) [circle,draw=black,fill=black, label=below:{$d_x$},inner sep=0pt, minimum size=.18cm]{};
\node (1) at (-1.5,-1) [circle,draw=black,fill=black, label=below:{$d_y$},inner sep=0pt, minimum size=.18cm]{};
\node (a) at (0,0) [circle,draw=black,fill=black, label=below:{$d_0$},inner sep=0pt, minimum size=.18cm]{};
\node (b) at (1.75,0) [circle,draw=black,fill=black, label=below:{$d_1$},inner sep=0pt, minimum size=.18cm]{};
\node (c) at (3.5,0) {$\boldsymbol{\cdots}$};
\node (d) at (5.25,0) [circle,draw=black,fill=black, label=below:{$d_{\ell}$},inner sep=0pt, minimum size=.18cm]{};
\draw (0)--(a);
\draw (1)--(a);
\draw (a) -- (b);
\draw (b) -- (c);
\draw (c) -- (d);
\end{tikzpicture}
\hspace{.1in}
\begin{tikzpicture}
\node (0) at (-1.5,1) [circle,draw=black,fill=black, label=below:{$r_x$},inner sep=0pt, minimum size=.18cm]{};
\node (1) at (-1.5,-1) [circle,draw=black,fill=black, label=below:{$r_y$},inner sep=0pt, minimum size=.18cm]{};
\node (a) at (0,0) [circle,draw=black,fill=black, label=below:{$r_0$},inner sep=0pt, minimum size=.18cm]{};
\node (b) at (1.75,0) [circle,draw=black,fill=black, label=below:{$r_1$},inner sep=0pt, minimum size=.18cm]{};
\node (c) at (3.5,0) {$\boldsymbol{\cdots}$};
\node (d) at (5.25,0) [circle,draw=black,fill=black, label=below:{$r_{\ell}$},inner sep=0pt, minimum size=.18cm]{};
\draw (0)--(a);
\draw (1)--(a);
\draw (a) -- (b);
\draw (b) -- (c);
\draw (c) -- (d);
\end{tikzpicture}
\end{center}
\caption{On the left we show the $\vd$-labeling for the graph $D_n$, and on the right we show the corresponding $\vr$-labeling.\label{fig:bidentlabels}}
\end{figure}

In order to determine the number of arithmetical structures on $D_n$, we first define in Section~\ref{sec:smooth} a notion of ``smooth'' arithmetical structure. Every arithmetical structure on $D_n$ with $d_x,d_y\geq2$ is associated to a unique smooth arithmetical structure on $D_n$ or a smaller bident (Lemma~\ref{lem:uniquesmoothancestor}). We use this to obtain the following theorem, which reduces the problem of enumerating arithmetical structures on bidents to that of enumerating smooth arithmetical structures on bidents.

\begin{namedthm*}{Theorem~\ref{thm:smoothtooverall}}
Let $n\geq 3$. The number of arithmetical structures on $D_n$ is given by
\[
\abs{\Arith(D_n)}=2C(n-2)+\sum_{m=4}^nB(n-3,n-m)\abs{\SArith(D_m)},
\]
where $C(n)$ is the $n$-th Catalan number, $B(n,k)=\frac{n-k+1}{n+1}\binom{n+k}{n}$ is a ballot number, and $\abs{\SArith(D_m)}$ is the number of smooth arithmetical structures on $D_m$.
\end{namedthm*}

In Section~\ref{sec:reduction}, we give a process for determining the number of smooth arithmetical structures, and hence the number of arithmetical structures, on $D_n$. We show that two parameters determine a smooth arithmetical structure on some bident, then find an expression for the number of vertices of this bident in terms of a function that measures the number of steps in a variant of the Euclidean algorithm. Analyzing this process in Section~\ref{sec:bounds}, we see that the number of smooth arithmetical structures on $D_n$ grows at the same rate as $n^3$ as $n$ increases (Theorem~\ref{thm:Xnbounds}). Together with Theorem~\ref{thm:smoothtooverall}, this yields the following theorem, which implies that the total number of arithmetical structures on $D_n$ grows at the same rate as the Catalan numbers as $n$ increases.

\begin{namedthm*}{Theorem~\ref{Thm:overallcount}}
For $n\geq 4$, we have that
\[
2 C(n-2)+ C(n-3) \leq \abs{\Arith(D_n)} < 2 C(n-2)+ 702  C(n-3).
\]
\end{namedthm*}

Finally, in Section~\ref{sec:critical}, we study critical groups of arithmetical structures on bidents. We show that these critical groups are always cyclic and obtain results about their orders. The maximal order of a critical group of an arithmetical structure on $D_n$ is $2n-5$ (Theorem~\ref{thm:criticalgp}), but there are values less than $2n-5$ that do not occur as orders of critical groups of arithmetical structures on $D_n$. Our main result about critical groups (Theorem~\ref{thm:critgporder}) determines, for each order $m$, the values of $n$ for which there is an arithmetical structure on $D_n$ with critical group of order $m$. This result completely characterizes the groups that occur as critical groups of arithmetical structures on $D_n$.

There are several open questions that remain related to this project, including finding a closed formula for the number of arithmetical structures on a bident. In addition, several of the techniques  in this paper should generalize to other families of graphs and be useful when studying arithmetical structures and their critical groups on graphs such as ``Y-graphs'' (graphs consisting of three paths that intersect at a common endpoint vertex) and ``I-graphs'' (graphs isomorphic to affine Dynkin diagrams $\widetilde{D}_n$).

\section{Smooth arithmetical structures}\label{sec:smooth}

In this section, we show how to count arithmetical structures on $D_n$ in terms of the number of ``smooth'' arithmetical structures on bidents. We focus primarily on arithmetical structures on $D_n$ with $d_x,d_y\geq2$, using the notation of Figure~\ref{fig:bidentlabels}, and show that all such structures can be obtained from a smooth arithmetical structure on some bident by a process of subdivision. As we will make precise in the proof of Theorem~\ref{thm:smoothtooverall}, enumerating arithmetical structures on $D_n$ with $d_x=1$ or $d_y=1$ reduces to enumerating arithmetical structures on path graphs, which has been done in~\cite{B17}.

\subsection{Definition and basic properties}

For $n\geq 3$, we say that an arithmetical structure $(\vd,\vr)$ on $D_n$ is \emph{smooth} if $d_x,d_y,d_1,d_2,\dotsc,d_{\ell}\geq2$; we denote by ${\SArith(D_n)}$ the set of smooth arithmetical structures on $D_n$. Note that this definition imposes no restriction on $d_0$; as we will see in Proposition~\ref{prop:centraldis1}, in fact smooth arithmetical structures on $D_n$ must have $d_0=1$. In Lemma~\ref{lem:smoothcharacterization}, we show that this definition is equivalent to the $r$-values strictly decreasing when moving away from the central vertex $v_0$.

As an example, consider the arithmetical structures on $D_4$, of which there are $14$, with $\vd$-vectors as follows:
\begin{center}
\begin{tabular}{lllll}
$\vd_1 = (1,1,3,1)$, & $\vd_2 = (3,3,1,3)$, &$\vd_3 = (6,3,1,2)$,&$\vd_4 = (3,6,1,2)$, &$\vd_5 = (6,2,1,3)$, \\
$\vd_6 = (2,6,1,3)$, &$\vd_7 = (3,2,1,6)$,&$\vd_8 = (2,3,1,6)$, &$\vd_9 = (2,1,2,2)$, &$\vd_{10} = (1,2,2,2)$, \\
$\vd_{11} = (2,2,2,1)$,&$\vd_{12} = (4,2,1,4)$,&$\vd_{13} = (2,4,1,4)$, &$\vd_{14} = (4,4,1,2)$,
\end{tabular}
\end{center}
and with corresponding $\vr$-vectors as follows:
\begin{center}
\begin{tabular}{lllll}
$\vr_1 = (1,1,1,1)$, & $\vr_2 = (1,1,3,1)$, &$\vr_3 = (1,2,6,3)$,&$\vr_4 = (2,1,6,3)$, &$\vr_5 = (1,3,6,2)$, \\
$\vr_6 = (3,1,6,2)$, &$\vr_7 = (2,3,6,1)$,&$\vr_8 = (3,2,6,1)$, &$\vr_9 = (1,2,2,1)$, &$\vr_{10} = (2,1,2,1)$, \\
$\vr_{11} = (1,1,2,2)$,&$\vr_{12} = (1,2,4,1)$,&$\vr_{13} = (2,1,4,1)$, &$\vr_{14} = (1,1,4,2)$.
\end{tabular}
\end{center}
Ten of these arithmetical structures on $D_4$ are smooth. The arithmetical structures $(\vd_1,\vr_1)$, $(\vd_9,\vr_9)$, $(\vd_{10},\vr_{10})$, and $(\vd_{11},\vr_{11})$ are not smooth  since their $\vd$-vectors have at least one of $d_x$, $d_y$, or $d_1$ equal to $1$.

\begin{lemma}\label{lem:smooth}
Let $n\geq 4$, and let $(\vd,\vr)$ be an arithmetical structure on $D_n$. The following conditions are equivalent:
\begin{enumerate}[ $(a)$]
\item \label{cond:no1s} $\dd_i\geq2$ for all $i\in\{1,2,\dotsc,\ell\}$;
\item \label{cond:shrinkinggaps} $\rr_0-\rr_1\geq\rr_1-\rr_2\geq\dotsb\geq\rr_{\ell-1}-\rr_\ell>0$;
\item \label{cond:decreasingrs}  $\rr_0>\rr_1>\dotsb>\rr_\ell$.
\end{enumerate}
\end{lemma}

\begin{proof}
We first show that $(a)$ implies $(b)$. Note that for $i\in\{1,2,\dotsc,\ell-1\}$, we have that $\dd_i\rr_i=\rr_{i-1}+\rr_{i+1}$. Therefore $\dd_i\rr_i-\rr_i=\rr_{i-1}+\rr_{i+1}-\rr_i$, or equivalently $(\dd_i-1)\rr_i-\rr_{i+1}=\rr_{i-1}-\rr_i$. If $\dd_i\geq2$, this means that $\rr_{i-1}-\rr_i\geq\rr_i-\rr_{i+1}$. We thus have that
\[
\rr_0-\rr_1\geq\rr_1-\rr_2\geq\dotsb\geq\rr_{\ell-1}-\rr_\ell.
\]
Since $r_{\ell-1}=d_\ell r_\ell \geq 2r_\ell$, we also have that $r_{\ell-1}-r_{\ell}>0$.

To see that $(b)$ implies $(c)$, observe that, since $\rr_{i-1}-\rr_{i}>0$ for all $i\in\{1,2,\dotsc,\ell\}$, it follows immediately that $\rr_0>\rr_1>\dotsb>\rr_\ell$.

Finally, we show that $(c)$ implies $(a)$.  Let $i\in\{1,2,\dotsc,\ell-1\}$. Since $\rr_{i}<\rr_{i-1}$, we have that $\rr_i<\rr_{i-1}+\rr_{i+1}=\dd_i\rr_i$. As $\dd_i$ is an integer, this means $\dd_i\geq2$ for all $i\in\{1,2,\dotsc,\ell-1\}$. Also, since $r_\ell<r_{\ell-1}=d_\ell r_\ell$, we must have that $d_\ell\geq2$.
\end{proof}

The following lemma characterizes smooth arithmetical structures in terms of their $\vr$-vectors.
\begin{lemma}\label{lem:smoothcharacterization}
An arithmetical structure $(\vd, \vr)$ on $D_n$ with $n\geq 4$ is smooth exactly when $r_x<r_0$, $r_y<r_0$, and $\rr_0>\rr_1>\dotsb>\rr_\ell$.
\end{lemma}

\begin{proof}
Since $d_x$ is an integer and $d_xr_x=r_0$, the condition $d_x\geq2$ is equivalent to $r_x<r_0$. Similarly, $d_y\geq2$ is equivalent to $r_y<r_0$. Lemma~\ref{lem:smooth} shows that the condition $d_i\geq2$ for all $i\in\{1,2,\dotsc,\ell\}$ is equivalent to the condition $\rr_0>\rr_1>\dotsb>\rr_\ell$.
\end{proof}

The next result shows that, while it is not a priori part of the definition, a smooth arithmetical structure on $D_n$ must have $d_0=1$.

\begin{proposition}\label{prop:centraldis1}
Every smooth arithmetical structure $(\vd, \vr)$ on $D_n$ with $n\geq3$ satisfies $d_0=1$.
\end{proposition}

\begin{proof}
Let $(\mathbf{d},\mathbf{r})$ be a smooth arithmetical structure on $D_n$. Since $d_x,d_y\geq2$, we must have $r_x\leq\frac{r_0}{2}$ and $r_y\leq\frac{r_0}{2}$. If $n=3$, we have $d_0r_0=r_x+r_y\leq r_0$, so the only possibility is $d_0=1$. When $n\geq4$, we have that $d_0r_0=r_x+r_y+r_1$, so therefore $r_0=\frac{1}{d_0}(r_x+r_y+r_1)$. If $d_0\geq2$, we would have $r_0\leq\frac{1}{2}(r_x+r_y+r_1)\leq\frac{r_0}{2}+\frac{r_1}{2}$. This would imply that $\frac{r_0}{2}\leq \frac{r_1}{2}$, but Lemma~\ref{lem:smoothcharacterization} tells us that $r_0>r_1$. Therefore we must have $d_0=1$.
\end{proof}

We use Proposition~\ref{prop:centraldis1} to show that appropriate values of $r_x$, $r_y$, and $r_0$ uniquely determine a smooth arithmetical structure on some bident.

\begin{proposition}\label{prop:abcuniqueness}
For every triple of integers $a,b,c\geq1$ with no common factor such that $a,b\mid c$ and $a,b<c$, there is a unique $n\geq3$ for which there is a smooth arithmetical structure on $D_n$ with $r_x=a$, $r_y=b$, and $r_0=c$.  Moreover, this smooth arithmetical structure on $D_n$ with $r_x=a$, $r_y=b$, and $r_0=c$ is unique.
\end{proposition}

\begin{proof}
Let $a$, $b$, and $c$ satisfy the given conditions, and suppose $c/a=c/b=2$. The fact that the three numbers share no common factor implies that $a=b=1$ and $c=2$. Setting $r_x=a$, $r_y=b$, and $r_0=c$ gives an arithmetical structure on $D_3$. Moreover, this does not give an arithmetical structure on $D_n$ for any $n\geq4$ since Proposition~\ref{prop:centraldis1} would require that $r_1=r_0-r_x-r_y=0$.

If $r_0/r_x$ and $r_0/r_y$ are not both $2$, Proposition~\ref{prop:centraldis1} says $r_1=r_0-r_x-r_y$. Whenever $r_i$ does not divide $r_{i-1}$, we must have that $r_{i+1}$ is the unique integer with $0<r_{i+1}<r_i$ such that $r_i\mid r_{i-1}+r_{i+1}$; Lemma~\ref{lem:smoothcharacterization} and the definition of arithmetical structure allow for no other possibility. We thus obtain a unique sequence $\{r_i\}$ that terminates with $r_{\ell}$, where $r_{\ell}\mid r_{\ell-1}$. We must therefore have  $n=\ell+3$, and this construction yields the unique smooth arithmetical structure on $D_n$ with $r_x=a$, $r_y=b$, and $r_0=c$.
\end{proof}

We note that a sequence $\{r_i\}$ with $0<r_{i+1}<r_i$ and $r_{i+1} \equiv -r_{i-1} \pmod {r_i}$ as in the proof of Proposition~\ref{prop:abcuniqueness} is what is referred to as a \emph{Euclidean chain} in~\cite{AB}.

We conclude this subsection by making the following observations applicable to both smooth and non-smooth arithmetical structures that will be used later.

\begin{lemma}\label{lem:gcd}
Let $n\geq4$, let $(\vd,\vr)$ be an arithmetical structure on $D_n$, and let $\ell=n-3$. Then
\begin{enumerate}[$(a)$]
\item $\gcd(r_x,r_y)=1$, and
\item $\gcd(r_0,r_1)=r_{\ell}$.
\end{enumerate}
\end{lemma}

\begin{proof}
First consider $(a)$. Let $c$ be a positive integer that divides $r_x$ and $r_y$. Since $r_x\mid r_0$, we have that $c\mid r_0$. Since $r_1=d_0r_0-r_x-r_y$, we have that $c\mid r_1$. Since $r_{i}=d_{i-1}r_{i-1}-r_{i-2}$ for all $i$ satisfying $2\leq i\leq \ell$, we have that $c\mid r_i$ for all $i$. Since $\vr$ is primitive, this means $c=1$. Therefore $\gcd(r_x,r_y)=1$.

To show $(b)$, first note that, for all $i$ satisfying $0\leq i\leq \ell-2$, we have that $r_i=d_{i+1}r_{i+1}-r_{i+2}$. Therefore
\[
\gcd(r_i,r_{i+1})=\gcd(d_{i+1}r_{i+1}-r_{i+2},r_{i+1})=\gcd(-r_{i+2},r_{i+1})=\gcd(r_{i+1},r_{i+2}).
\]
Repeatedly applying this gives that $\gcd(r_0,r_1)=\gcd(r_{\ell-1},r_{\ell})$. Since $r_{\ell-1}=d_{\ell}r_{\ell}$, we have that $\gcd(r_{\ell-1},r_{\ell})=r_{\ell}$. The result follows.
\end{proof}

\subsection{Smoothing and subdivision}\label{sec:subdivide}

We now discuss the complementary operations of \emph{smoothing} and \emph{subdivision} of arithmetical structures on bidents. At vertices of degree $2$, our notions of smoothing and subdivision are the same as those found in~\cite{B17}. However, we also allow smoothing at vertices of degree $1$ and subdivision to create new vertices of degree $1$. For the convenience of the reader, we describe the notions from~\cite{B17} that we use, as well as the aforementioned extension to degree 1 vertices. The proofs of many of the results in this subsection and the next are generalizations of the proofs given in that paper.  We include them here both to highlight the differences and to keep this article self-contained.

\subsubsection{Process of smoothing}

Let $n\geq 4$, and let $(\vd,\vr)$ be an arithmetical structure on $D_n$. If $d_i=1$ for some $i\in\{1,2,\dotsc,\ell-1\}$, we can obtain a new arithmetical structure $(\vd', \vr')$ on $D_{n-1}$ by essentially removing the vertex $v_i$ and  leaving the $\vr$-labeling unchanged for the remaining vertices, while adjusting the $\vd$-labeling in the appropriate manner. To be precise, we define vectors $\vr'$ and $\vd'$ of length $n-1$ as follows:
\begin{align*}
r'_j&=\begin{cases}
r_j&\text{for }j\in\{x,y,0,1,\dotsc,i-1\}\\
r_{j+1}&\text{for }j\in\{i,i+1,\dotsc,\ell-1\},
\end{cases}\\
d'_j&=\begin{cases}
d_j&\text{for }j\in\{x,y,0,1,\dotsc,i-2\}\\
d_j-1&\text{for }j=i-1\\
d_{j+1}-1&\text{for }j=i\\
d_{j+1}&\text{for }j\in\{i+1,i+2,\dotsc,\ell-1\}.
\end{cases}
\end{align*}
It is straightforward to check that $(\vd',\vr')$ satisfies the defining equations of an arithmetical structure on $D_{n-1}$. To show that it is an arithmetical structure, it remains only to verify that $\vd'\in\mathbb{Z}^{n-1}_{>0}$, which follows from~\cite[Lemma 6]{B17}. We refer to the operation described above that takes in an arithmetical structure on $D_n$ and returns one on $D_{n-1}$ as \emph{smoothing at vertex $v_i$} or \emph{smoothing at position $i$}.
An example of this smoothing process is shown in Figure~\ref{fig:smooth1}.

\begin{figure}
\centering
\resizebox{\textwidth}{!}{
\begin{tabular}{cc}
\subcaptionbox{$\vr$-labeling before smoothing}{\fbox{
\begin{tikzpicture}
\node (0) at (-1.5,1) [circle,draw=black,fill=black, label=above:{$7$},inner sep=0pt, minimum size=.18cm]{};
\node (1) at (-1.5,-1) [circle,draw=black,fill=black,label=below:{$2$},inner sep=0pt, minimum size=.18cm]{};
\node (a) at (0,0) [circle,draw=black,fill=black, label=above:{$14$},inner sep=0pt, minimum size=.18cm]{};
\node (b) at (1.75,0) [circle,draw=black,fill=black, label=above:{$5$},inner sep=0pt, minimum size=.18cm]{};
\node (c) at (3.5,0) [circle,draw=black,fill=black, label=above:{$6$},inner sep=0pt, minimum size=.18cm]{};
\node (d) at (5.25,0) [circle,draw=black,fill=black, label=above:{$1$},inner sep=0pt, minimum size=.18cm]{};
\draw (0)--(a);
\draw (1)--(a);
\draw (a) -- (b);
\draw (b) -- (c);
\draw (c) -- (d);
\end{tikzpicture}}}
&
\subcaptionbox{$\vd$-labeling before smoothing}{\fbox{%
\begin{tikzpicture}
\node (0) at (-1.5,1) [circle,draw=black,fill=black, label=above:{$2$},inner sep=0pt, minimum size=.18cm]{};
\node (1) at (-1.5,-1) [circle,draw=black,fill=black,label=below:{$7$},inner sep=0pt, minimum size=.18cm]{};
\node (a) at (0,0) [circle,draw=black,fill=black, label=above:{$1$},inner sep=0pt, minimum size=.18cm]{};
\node (b) at (1.75,0) [circle,draw=black,fill=black, label=above:{$4$},inner sep=0pt, minimum size=.18cm]{};
\node (c) at (3.5,0) [circle,draw=black,fill=black, label=above:{$1$},inner sep=0pt, minimum size=.18cm]{};
\node (d) at (5.25,0) [circle,draw=black,fill=black, label=above:{$6$},inner sep=0pt, minimum size=.18cm]{};
\draw (0)--(a);
\draw (1)--(a);
\draw (a) -- (b);
\draw (b) -- (c);
\draw (c) -- (d);
\end{tikzpicture}}}
\\[20pt]
\subcaptionbox{$\vr$-labeling after smoothing at $v_2$}{\fbox{%
\begin{tikzpicture}
\node (0) at (-1.5,1) [circle,draw=black,fill=black, label=above:{$7$},inner sep=0pt, minimum size=.18cm]{};
\node (1) at (-1.5,-1) [circle,draw=black,fill=black,label=below:{$2$},inner sep=0pt, minimum size=.18cm]{};
\node (a) at (0,0) [circle,draw=black,fill=black, label=above:{$14$},inner sep=0pt, minimum size=.18cm]{};
\node (b) at (1.75,0) [circle,draw=black,fill=black, label=above:{$5$},inner sep=0pt, minimum size=.18cm]{};
\node (d) at (5.25,0) [circle,draw=black,fill=black, label=above:{$1$},inner sep=0pt, minimum size=.18cm]{};
\draw (0)--(a);
\draw (1)--(a);
\draw (a) -- (b);
\draw (b) -- (d);
\end{tikzpicture}}}
&
\subcaptionbox{$\vd$-labeling after smoothing at $v_2$}{\fbox{%
\begin{tikzpicture}
\node (0) at (-1.5,1) [circle,draw=black,fill=black, label=above:{$2$},inner sep=0pt, minimum size=.18cm]{};
\node (1) at (-1.5,-1) [circle,draw=black,fill=black,label=below:{$7$},inner sep=0pt, minimum size=.18cm]{};
\node (a) at (0,0) [circle,draw=black,fill=black, label=above:{$1$},inner sep=0pt, minimum size=.18cm]{};
\node (b) at (1.75,0) [circle,draw=black,fill=black, label=above:{$3$},inner sep=0pt, minimum size=.18cm]{};
\node (d) at (5.25,0) [circle,draw=black,fill=black, label=above:{$5$},inner sep=0pt, minimum size=.18cm]{};
\draw (0)--(a);
\draw (1)--(a);
\draw (a) -- (b);
\draw (b) -- (d);
\end{tikzpicture}}}
\end{tabular}
}
\caption{Pictured above is an arithmetical structure on $D_6$, represented in (a) by its $\vr$-labeling and in (b) by its $\vd$-labeling. Since $d_2=1$, we can smooth at vertex $v_2$ to obtain the arithmetical structure on $D_5$ represented in (c) by its $\vr$-labeling and in (d) by its $\vd$-labeling.\label{fig:smooth1}}
\end{figure}
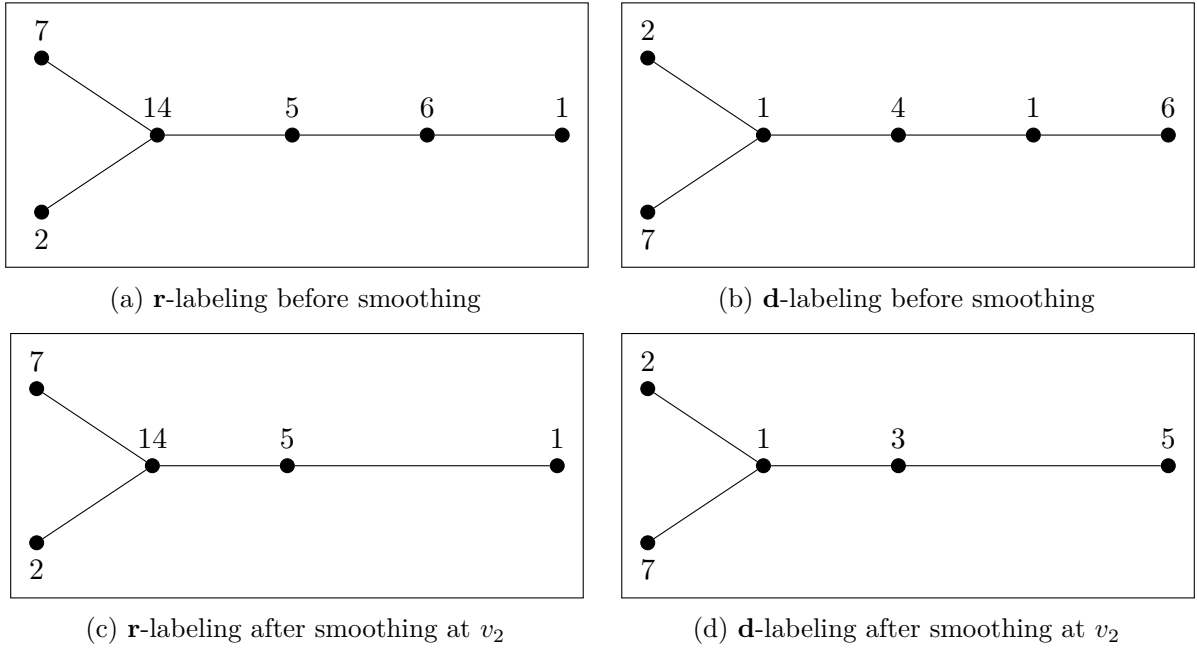

Now, let us describe how we can extend this smoothing operation to vertices of degree 1. There are three degree $1$ vertices of $D_n$: the one at the end of the tail and the two at the end of the ``prongs'' of the bident. Let us first consider the vertex $v_\ell$ at the end of the tail. If $d_\ell=1$, we can obtain a new arithmetical structure $(\vd',\vr')$ on $D_{n-1}$ by removing vertex $v_\ell$ and decreasing $d_{\ell-1}$ by $1$. That is, define $d_i'=d_i$ when $i\neq \ell-1$ and take $d'_{\ell-1}=d_{\ell-1}-1$. The corresponding $\vr$-labeling remains unchanged under this smoothing (except that $r_\ell$ no longer appears). We refer to this operation as \emph{smoothing at vertex $v_\ell$} or \emph{smoothing at position $\ell$}. See Figure~\ref{fig:smooth2} for an example of this operation.

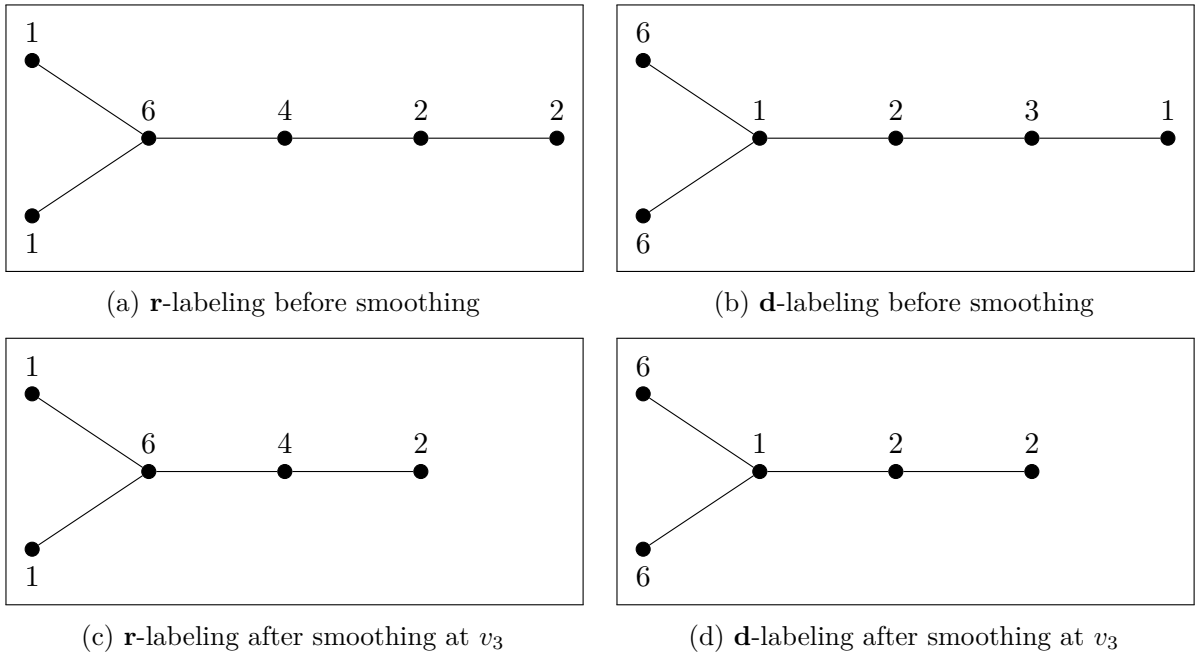
\begin{figure}
\centering
\resizebox{\textwidth}{!}{
\begin{tabular}{cc}
\subcaptionbox{$\vr$-labeling before smoothing}{\fbox{%
\begin{tikzpicture}
\node (0) at (-1.5,1) [circle,draw=black,fill=black, label=above:{$1$},inner sep=0pt, minimum size=.18cm]{};
\node (1) at (-1.5,-1) [circle,draw=black,fill=black,label=below:{$1$},inner sep=0pt, minimum size=.18cm]{};
\node (a) at (0,0) [circle,draw=black,fill=black, label=above:{$6$},inner sep=0pt, minimum size=.18cm]{};
\node (b) at (1.75,0) [circle,draw=black,fill=black, label=above:{$4$},inner sep=0pt, minimum size=.18cm]{};
\node (c) at (3.5,0) [circle,draw=black,fill=black, label=above:{$2$},inner sep=0pt, minimum size=.18cm]{};
\node (d) at (5.25,0) [circle,draw=black,fill=black, label=above:{$2$},inner sep=0pt, minimum size=.18cm]{};
\draw (0)--(a);
\draw (1)--(a);
\draw (a) -- (b);
\draw (b) -- (c);
\draw (c) -- (d);
\end{tikzpicture}}}
&
\subcaptionbox{$\vd$-labeling before smoothing}{\fbox{%
\begin{tikzpicture}
\node (0) at (-1.5,1) [circle,draw=black,fill=black, label=above:{$6$},inner sep=0pt, minimum size=.18cm]{};
\node (1) at (-1.5,-1) [circle,draw=black,fill=black,label=below:{$6$},inner sep=0pt, minimum size=.18cm]{};
\node (a) at (0,0) [circle,draw=black,fill=black, label=above:{$1$},inner sep=0pt, minimum size=.18cm]{};
\node (b) at (1.75,0) [circle,draw=black,fill=black, label=above:{$2$},inner sep=0pt, minimum size=.18cm]{};
\node (c) at (3.5,0) [circle,draw=black,fill=black, label=above:{$3$},inner sep=0pt, minimum size=.18cm]{};
\node (d) at (5.25,0) [circle,draw=black,fill=black, label=above:{$1$},inner sep=0pt, minimum size=.18cm]{};
\draw (0)--(a);
\draw (1)--(a);
\draw (a) -- (b);
\draw (b) -- (c);
\draw (c) -- (d);
\end{tikzpicture}}}
\\[20pt]
\subcaptionbox{$\vr$-labeling after smoothing at $v_3$}{\fbox{%
\begin{tikzpicture}
\node (0) at (-1.5,1) [circle,draw=black,fill=black, label=above:{$1$},inner sep=0pt, minimum size=.18cm]{};
\node (1) at (-1.5,-1) [circle,draw=black,fill=black,label=below:{$1$},inner sep=0pt, minimum size=.18cm]{};
\node (a) at (0,0) [circle,draw=black,fill=black, label=above:{$6$},inner sep=0pt, minimum size=.18cm]{};
\node (b) at (1.75,0) [circle,draw=black,fill=black, label=above:{$4$},inner sep=0pt, minimum size=.18cm]{};
\node (d) at (3.5,0) [circle,draw=black,fill=black, label=above:{$2$},inner sep=0pt, minimum size=.18cm]{};
\node (c) at (5.25,0) [circle,draw=white,fill=white, label=above:{\phantom{$2$}},inner sep=0pt, minimum size=.18cm]{};
\draw (0)--(a);
\draw (1)--(a);
\draw (a) -- (b);
\draw (b) -- (d);
\end{tikzpicture}}}
&
\subcaptionbox{$\vd$-labeling after smoothing at $v_3$}{\fbox{%
\begin{tikzpicture}
\node (0) at (-1.5,1) [circle,draw=black,fill=black, label=above:{$6$},inner sep=0pt, minimum size=.18cm]{};
\node (1) at (-1.5,-1) [circle,draw=black,fill=black,label=below:{$6$},inner sep=0pt, minimum size=.18cm]{};
\node (a) at (0,0) [circle,draw=black,fill=black, label=above:{$1$},inner sep=0pt, minimum size=.18cm]{};
\node (b) at (1.75,0) [circle,draw=black,fill=black, label=above:{$2$},inner sep=0pt, minimum size=.18cm]{};
\node (d) at (3.5,0) [circle,draw=black,fill=black, label=above:{$2$},inner sep=0pt, minimum size=.18cm]{};
\node (c) at (5.25,0) [circle,draw=white,fill=white, label=above:{\phantom{$2$}},inner sep=0pt, minimum size=.18cm]{};
\draw (0)--(a);
\draw (1)--(a);
\draw (a) -- (b);
\draw (b) -- (d);
\end{tikzpicture}}}
\end{tabular}
}
\caption{Pictured above is an arithmetical structure on $D_6$, represented in (a) by its $\vr$-labeling and in (b) by its $\vd$-labeling. Since $d_3=1$, we can smooth at vertex $v_3$ (the end of the tail) to obtain the arithmetical structure on $D_5$ represented in (c) by its $\vr$-labeling and in (d) by its $\vd$-labeling.\label{fig:smooth2}}
\end{figure}

Finally, we can also smooth at the vertex at the end of one of the ``prongs'' of the bident when $d_y=1$ (or $d_x=1$). In this case, we can find an arithmetical structure $(\vd', \vr')$ by taking the entries of $\vd'$ and $\vr'$ to be equal to the corresponding ones in $\vd$ and $\vr$ except that $d_0'=d_0-1$. In this case, the resulting arithmetical structure is an arithmetical structure on a path graph and not a bident. It is again straightforward to check that $(\vd',\vr')$ is indeed an arithmetical structure. We call this process \emph{smoothing at vertex $v_y$ \textup{(}or $v_x$\textup{)}} or \emph{smoothing at position $y$ \textup{(}or $x$\textup{)}}. If $d_x=d_y=1$, one could perform this operation at both vertices $v_x$ and $v_y$ and obtain the arithmetical structure $(\vd'', \vr'')$ on the remaining path graph obtained by taking $d_0''=d_0-2$, and leaving all other corresponding entries unchanged.

In each case, we refer to the new arithmetical structure $(\vd',\vr')$ on a graph with fewer vertices as a \emph{smoothing} of $(\vd,\vr)$. Note that it is possible to perform a smoothing operation on any arithmetical structure on $D_n$ that is not smooth, i.e.\ those arithmetical structures for which there is some $i\in\{x,y,1,2,\dotsc,\ell\}$ with $d_i=1$. Thus, smooth arithmetical structures on $D_n$ are precisely those on which no smoothing operation can be performed. If an arithmetical structure $(\vd',\vr')$ on $D_m$ can be obtained from an arithmetical structure $(\vd,\vr)$ on $D_n$ by a sequence of smoothing operations, we say that $(\vd',\vr')$ is an \emph{ancestor} of $(\vd,\vr)$. (Indeed, it is an ancestor in the poset of arithmetical structures under the ordering induced by this operation.)

\begin{lemma}\label{lem:uniquesmoothancestor}
Every arithmetical structure on $D_n$ with $d_x,d_y\geq2$ has a unique smooth ancestor on $D_m$ for some $m$ satisfying $3\leq m\leq n$.
\end{lemma}

\begin{proof}
Let $(\vd,\vr)$ be an arithmetical structure on $D_n$ with $d_x,d_y\geq2$. If $d_i=1$ for some $i\in\{1,2,\dotsc,\ell\}$, perform a smoothing operation at vertex $v_i$. Repeat until obtaining an arithmetical structure $(\vd',\vr')$ on $D_{n-s}$ (where $s$ is the number of smoothing operations that have been performed) with $d'_i>1$ for all $i\geq 1$. (In the worst case, this process will terminate when this condition is satisfied vacuously, i.e.\ when we are left with an arithmetical structure on $D_3$.) Note that each step eliminates a vertex $v_i$ where $r_{i-1},r_{i+1}<r_i$. Therefore the remaining sequence $r'_0,r'_1,\dotsc,r'_{\ell-s}$ is the maximal decreasing subsequence of $r_0,r_1,\dotsc,r_{\ell}$ (since the entries of the $\vr$ vector entries remain unchanged, except via deletion, under smoothing), and hence is uniquely determined. Moreover, $d_x$ and $d_y$ are unchanged by these operations. The arithmetical structure $(\vd',\vr')$ is thus the unique smooth ancestor of $(\vd,\vr)$.
\end{proof}

\subsubsection{Process of subdivision}

We now discuss \emph{subdivision}, which is the inverse operation of smoothing. Given an arithmetical structure $(\vd,\vr)$ on $D_n$, we obtain an arithmetical structure $(\vd',\vr')$ on $D_{n+1}$ by adding a vertex in the tail of the graph $D_n$, assigning it a $\vd$-label of 1 and an $\vr$-label given by the sum of the $\vr$-labels of its neighboring vertices. Adding the vertex at the end of the tail is also allowed, in which case its corresponding entry in $\vr$ is equal to that of its neighbor. More precisely, for $i$ with $1\leq i\leq \ell$, we define vectors $\vr'$ and $\vd'$ of length $n+1$ as follows:
\begin{align*}
r'_j&=\begin{cases}
r_j&\text{for }j\in\{x,y,0,1,\dotsc,i-1\}\\
r_{j-1}+r_j&\text{for }j=i\\
r_{j-1}&\text{for }j\in\{i+1,i+2,\dotsc,\ell+1\},
\end{cases} \\
d'_j&=\begin{cases}
d_j&\text{for }j\in\{x,y,0,1,\dotsc,i-2\}\\
d_j+1&\text{for }j=i-1\\
1&\text{for }j=i\\
d_{j-1}+1&\text{for }j=i+1\\
d_{j-1}&\text{for }j\in\{i+2,i+3,\dotsc,\ell+1\},
\end{cases}
\end{align*}
and for $i=\ell+1$, we define $\vr'$ and $\vd'$ as follows:
\begin{align*}
r'_j&=\begin{cases}
r_j&\text{for }j\in\{x,y,0,1,\dotsc,\ell\}\\
r_{j-1}&\text{for }j=\ell+1,
\end{cases} \\
d'_j&=\begin{cases}
d_j&\text{for }j\in\{x,y,0,1,\dotsc,\ell-1\}\\
d_j+1&\text{for }j=\ell\\
1&\text{for }j=\ell+1.
\end{cases}
\end{align*}
In both cases, it is straightforward to check that $(\vd',\vr')$ is an arithmetical structure on $D_{n+1}$. We call $(\vd',\vr')$ the \emph{subdivision at position $i$} of $(\vd,\vr)$. An example of subdivision in the interior of the tail is shown in Figure~\ref{fig:subd1}, and subdivision at the end of the tail is the inverse of the smoothing operation shown in Figure~\ref{fig:smooth2}.

The subdivision operation that is inverse to smoothing at $v_y$ begins with an arithmetical structure on a path graph and adds a new vertex $v_y$, connecting it to $v_0$ by a single edge and setting $r'_y=r_0$, $d'_y=1$, and $d'_0=d_0+1$ while leaving the other $r$-values and $d$-values unchanged. We call this operation \emph{subdivision at position $y$}. We can similarly define subdivision at position $x$. More generally, we could define a subdivision operation on an arithmetical structure $(\vd,\vr)$ on any graph by adding a new vertex $v_y$, connecting it by a single edge to any other vertex $v_0$ in the graph, and setting $r'_y=r_0$, $d'_y=1$, and $d'_0=d_0+1$ while leaving the other $r$-values and $d$-values unchanged.

If an arithmetical structure $(\vd,\vr)$ on $D_n$ can be obtained from an arithmetical structure $(\vd',\vr')$ on $D_m$ by a sequence of subdivision operations,  we say that $(\vd',\vr')$ is a \emph{descendant} of $(\vd,\vr)$. Note that every smoothing operation has an inverse subdivision operation and vice versa. Therefore  $(\vd,\vr)$ is a descendant of $(\vd',\vr')$ if and only if $(\vd',\vr')$ is an ancestor of $(\vd,\vr)$.

\subsection{Subdivision sequences and counting}

Let $(\vd^0,\vr^0)$ be an arithmetical structure on $D_m$, with $3\leq m\leq n$. We say that a sequence of positive integers $\mathbf{b}=(b_1,b_2,\dotsc,b_{n-m})$ is a \emph{valid subdivision sequence} for $(\vd^0,\vr^0)$ if its entries satisfy $1\leq b_i\leq m-3+i$. We inductively define an arithmetical structure $\Sub((\vd^0,\vr^0),\mathbf{b})$ on $D_n$ from this sequence $\mathbf{b}$ as follows. Let $(\vd^i,\vr^i)$ be the arithmetical structure on $D_{m+i}$ obtained from the arithmetical structure $(\vd^{i-1},\vr^{i-1})$ on $D_{m+i-1}$ by subdividing at position $b_i$, which we can do as long as $1 \le b_i \le m-3+i$.  We then define
\[
\Sub((\vd^0,\vr^0),\mathbf{b})\coloneqq(\vd^{n-m},\vr^{n-m}).
\]
If $m=n$, then $\mathbf{b}$ is the empty sequence and $\Sub((\vd^0,\vr^0),\mathbf{b})=(\vd^0,\vr^0)$. If $m=3$, then the condition requires that $b_1=1$, meaning we must first subdivide at position $1$ to obtain an arithmetical structure on $D_4$. Note that the descendants of $(\vd^0,\vr^0)$ are exactly those arithmetical structures of the form $\Sub((\vd^0,\vr^0),\mathbf{b})$ for some such sequence $\mathbf{b}$.

\begin{lemma}\label{lem:descendantequivalence}
Let $3\leq m\leq n$, let $(\vd^0,\vr^0)$ be an arithmetical structure on $D_m$, and let $\mathbf{b}=(b_1,b_2,\dotsc,b_{n-m})$ be a valid subdivision sequence for $(\vd^0,\vr^0)$. Suppose $j$ is a positive integer satisfying $1\leq j<n-m$ with $b_j>b_{j+1}$. Define $\mathbf{b}'=(b'_1,b'_2,\dotsc,b'_{n-m})$ by
\[
b'_i=\begin{cases}
b_{j+1}&\text{for }i=j\\
b_j+1&\text{for }i=j+1\\
b_i&\text{otherwise.}
\end{cases}
\]
Then $\Sub((\vd^0,\vr^0),\mathbf{b})=\Sub((\vd^0,\vr^0),\mathbf{b'})$.
\end{lemma}

This lemma is the same as~\cite[Lemma 13]{B17} except that it also allows for subdivision at vertex $v_{\ell + 1}$, and the proof follows directly from the definitions. As an example, observe that the arithmetical structure shown in Figure~\ref{fig:subd1}(d) can be obtained from the arithmetical structure shown in Figure~\ref{fig:subd1}(a) using any of $\mathbf{b}=(2,2,1)$, $(2,1,3)$, or $(1,3,3)$. Lemma~\ref{lem:descendantequivalence} implies that the order of subdivision along the tail does not matter unless the subdivisions are adjacent to each other. The following lemma and its proof are similar to~\cite[Proposition 14]{B17}.

\begin{lemma}\label{lem:ancestor}
Fix an arithmetical structure $(\vd^0,\vr^0)$ on $D_m$ with $d^0_i\geq2$ for all $i\in\{1,2,\dotsc,m-3\}$. There is a bijection between arithmetical structures on $D_n$ that are descendants of $(\vd^0,\vr^0)$ and valid subdivision sequences $\mathbf{b}=(b_1,b_2,\dotsc,b_{n-m})$ that additionally satisfy $b_i\leq b_{i+1}$ for all $i$.
\end{lemma}

\begin{proof}
If $(\vd,\vr)$ is an arithmetical structure on $D_n$ that is a descendant of $(\vd^0,\vr^0)$, then $(\vd,\vr)=\Sub((\vd^0,\vr^0),\mathbf{b}')$ for some $\mathbf{b}'=(b'_1,b'_2,\dotsc b'_{n-m})$ satisfying $1\leq b'_i\leq m-3+i$ for all~$i$. Repeatedly applying Lemma~\ref{lem:descendantequivalence} then shows that $(\vd,\vr)$ is equal to $\Sub((\vd^0,\vr^0),\mathbf{b})$ for some sequence~$\mathbf{b}$ of the desired type.

The sequence $\mathbf{b}$ has the property that, at each stage of the subdivision, $b_i$ is the largest value of $j$ such that $d^i_j=1$. Starting with an arithmetical structure $(\vd,\vr)$ on $D_n$ that is a descendant of $(\vd^0,\vr^0)$ and repeatedly subdividing at position $j$, where $j$ is the largest number with $d_j=1$, therefore shows how to recover $\mathbf{b}$ and implies there is a unique such sequence for each descendant of $(\vd^0,\vr^0)$.
\end{proof}

\begin{figure}
\centering
\begin{tabular}{cc}
\subcaptionbox{original structure on $D_4$}{\fbox{%
\begin{tikzpicture}[scale=0.84]
\node (0) at (-1.5,1) [circle,draw=black,fill=black, label=above:{$3$},inner sep=0pt, minimum size=.18cm]{};
\node (1) at (-1.5,-1) [circle,draw=black,fill=black,label=below:{$1$},inner sep=0pt, minimum size=.18cm]{};
\node (a) at (0,0) [circle,draw=black,fill=black, label=above:{$6$},inner sep=0pt, minimum size=.18cm]{};
\node (b) at (1.75,0) [circle,draw=black,fill=black, label=above:{$2$},inner sep=0pt, minimum size=.18cm]{};
\node (c) at (3.5,0) [circle,draw=white,fill=white, label=above:{\phantom{$2$}},inner sep=0pt, minimum size=.18cm]{};
\node (d) at (5.25,0) [circle,draw=white,fill=white, label=above:{\phantom{$2$}},inner sep=0pt, minimum size=.18cm]{};
\node (e) at (7,0) [circle,draw=white,fill=white, label=above:{\phantom{$2$}},inner sep=0pt, minimum size=.18cm]{};
\draw (0)--(a);
\draw (1)--(a);
\draw (a) -- (b);
% \draw (b) -- (c);
%  \draw (c) -- (d);
% \draw (d) -- (e);
\end{tikzpicture}}}
&
\subcaptionbox{after the first step}{\fbox{%
\begin{tikzpicture}[scale=0.84]
\node (0) at (-1.5,1) [circle,draw=black,fill=black, label=above:{$3$},inner sep=0pt, minimum size=.18cm]{};
\node (1) at (-1.5,-1) [circle,draw=black,fill=black,label=below:{$1$},inner sep=0pt, minimum size=.18cm]{};
\node (a) at (0,0) [circle,draw=black,fill=black, label=above:{$6$},inner sep=0pt, minimum size=.18cm]{};
\node (b) at (1.75,0) [circle,draw=black,fill=black, label=above:{$8$},inner sep=0pt, minimum size=.18cm]{};
\node (c) at (3.5,0) [circle,draw=black,fill=black, label=above:{$2$},inner sep=0pt, minimum size=.18cm]{};
\node (d) at (5.25,0) [circle,draw=white,fill=white, label=above:{\phantom{$2$}},inner sep=0pt, minimum size=.18cm]{};
\node (e) at (7,0) [circle,draw=white,fill=white, label=above:{\phantom{$2$}},inner sep=0pt, minimum size=.18cm]{};
\draw (0)--(a);
\draw (1)--(a);
\draw (a) -- (b);
\draw (b) -- (c);
%  \draw (c) -- (d);
% \draw (d) -- (e);
\end{tikzpicture}}}
\\[20pt]
\subcaptionbox{after the second step}{\fbox{%
\begin{tikzpicture}[scale=0.84]
\node (0) at (-1.5,1) [circle,draw=black,fill=black, label=above:{$3$},inner sep=0pt, minimum size=.18cm]{};
\node (1) at (-1.5,-1) [circle,draw=black,fill=black,label=below:{$1$},inner sep=0pt, minimum size=.18cm]{};
\node (a) at (0,0) [circle,draw=black,fill=black, label=above:{$6$},inner sep=0pt, minimum size=.18cm]{};
\node (b) at (1.75,0) [circle,draw=black,fill=black, label=above:{$8$},inner sep=0pt, minimum size=.18cm]{};
\node (c) at (3.5,0) [circle,draw=black,fill=black, label=above:{$2$},inner sep=0pt, minimum size=.18cm]{};
\node (d) at (5.25,0) [circle,draw=black,fill=black, label=above:{$2$},inner sep=0pt, minimum size=.18cm]{};
\node (e) at (7,0) [circle,draw=white,fill=white, label=above:{\phantom{$2$}},inner sep=0pt, minimum size=.18cm]{};
\draw (0)--(a);
\draw (1)--(a);
\draw (a) -- (b);
\draw (b) -- (c);
\draw (c) -- (d);
% \draw (d) -- (e);
\end{tikzpicture}}}
&
\subcaptionbox{after the last step}{\fbox{%
\begin{tikzpicture}[scale=0.84]
\node (0) at (-1.5,1) [circle,draw=black,fill=black, label=above:{$3$},inner sep=0pt, minimum size=.19cm]{};
\node (1) at (-1.5,-1) [circle,draw=black,fill=black,label=below:{$1$},inner sep=0pt, minimum size=.19cm]{};
\node (a) at (0,0) [circle,draw=black,fill=black, label=above:{$6$},inner sep=0pt, minimum size=.19cm]{};
\node (b) at (1.75,0) [circle,draw=black,fill=black, label=above:{$8$},inner sep=0pt, minimum size=.19cm]{};
\node (c) at (3.5,0) [circle,draw=black,fill=black, label=above:{$2$},inner sep=0pt, minimum size=.19cm]{};
\node (d) at (5.25,0) [circle,draw=black,fill=black, label=above:{$4$},inner sep=0pt, minimum size=.19cm]{};
\node (e) at (7,0) [circle,draw=black,fill=black, label=above:{$2$},inner sep=0pt, minimum size=.19cm]{};

\draw (0)--(a);
\draw (1)--(a);
\draw (a) -- (b);
\draw (b) -- (c);
\draw (c) -- (d);
\draw (d) -- (e);
\end{tikzpicture}}}
\end{tabular}
\caption{Starting with the arithmetical structure illustrated in (a) via the $\vr$-labeling of the graph $D_4$, we use the sequence $\mathbf{b} = (1, 3, 3)$ to obtain the arithmetical structure illustrated in (d) on $D_7$. The steps are shown as follows. To obtain the structure in (b), subdivide the structure in (a) at position 1. To obtain the structure in (c), subdivide the structure in (b) at position 3 (the end of the tail). Finally, to obtain the resulting structure in (d), subdivide the structure pictured in (c) at position 3.\label{fig:subd1}}
\end{figure}
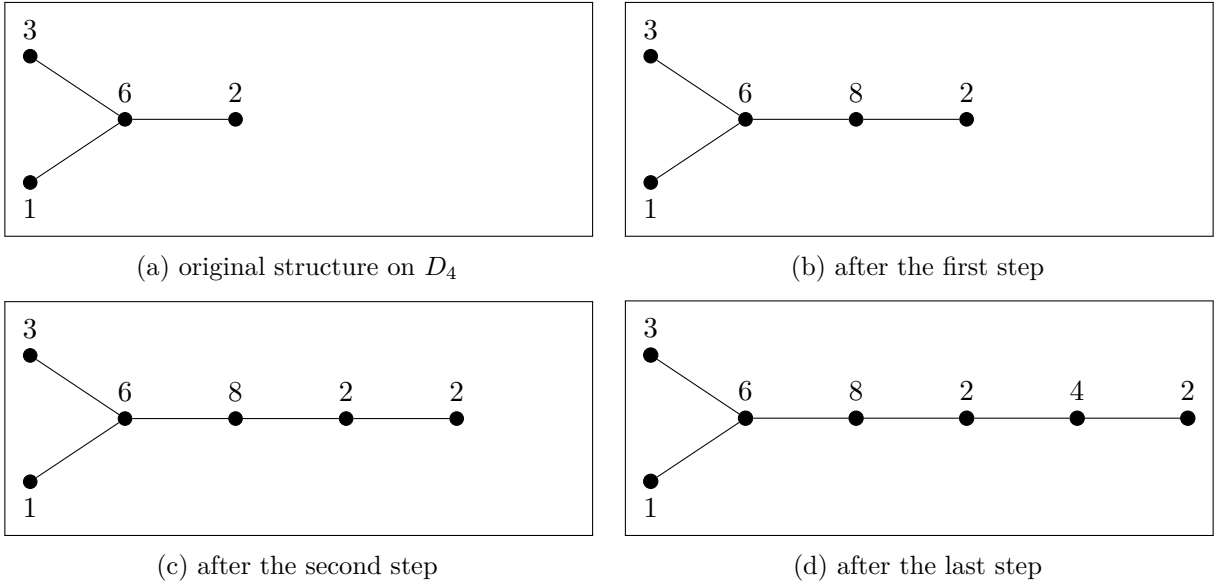

Let $C(n)$ denote the \emph{Catalan numbers}~\cite[A009766]{OEIS}, defined for all $n\geq0$ by the formula
\[
C(n)=\frac{1}{n+1}\binom{2n}{n},
\]
and let $B(n,k)$ denote the so-called \emph{ballot numbers}, defined for all $n \ge k \ge 0$ by the formula
\[
B(n,k) = \frac{n-k+1}{n+1}\binom{n+k}{n}.
\]
The ballot numbers are a generalization of the Catalan numbers that were first studied by Carlitz~\cite{Carlitz}. They can alternatively be defined by setting $B(n,0)=1$ for all $n$, $B(n,k)=0$ for all $k>n$ and otherwise $B(n,k) = B(n,k-1) +B(n-1,k)$. The ballot numbers will be used to enumerate nondecreasing valid subdivision sequences, but we first establish the following lemma, which is the analogue of~\cite[Lemma 15]{B17}.

\begin{lemma}\label{ballot}
For any $n\geq 1$ and $n \ge k \ge 0$, the number of nondecreasing sequences $(b_1,b_2,\dotsc, b_n)$ with $b_i\leq i$ for all $i$, such that additionally $b_j=1$ for all $j\leq k$, is equal to $B(n,n-k)$.
\end{lemma}

\begin{proof}
Let $\BB(n,k)$ denote the number of nondecreasing sequences with $b_i\leq i$ and beginning with at least $k$ leading ones.  We wish to show that $\BB(n,k)=B(n,n-k)$.  We will do so by showing that both satisfy the same initial conditions and the same recurrence relation.  In particular, we will show that $\BB(n,n)=1$ for all $n$, that $\BB(n,k)=0$ if $k<0$, and that $\BB(n,k)= \BB(n,k+1) + \BB(n-1,k-1)$.

The first two statements are clear, as there is a unique sequence of length $n$ with $n$ leading ones and there are no sequences with a negative number of leading ones. To see the third statement, note that the set of sequences of length $n$ with at least $k$ leading ones can be decomposed into two disjoint sets: those with at least $k+1$ leading ones (enumerated by $\BB(n,k+1)$) and those with exactly $k$ leading ones. If a sequence has exactly $k$ leading ones then it follows that $b_{k+1}>1$.  In particular, one can obtain a sequence of length $n-1$ with at least $k-1$ leading ones that we will call $b'$ by deleting the $k$-th occurrence of 1 in $b$ and subtracting 1 from each $b_i$ for $k+1\leq i \leq n$.  This process is invertible, which argues that the number of such sequences is $\BB(n-1,k-1)$.  In particular, we have shown that $\BB(n,k)=\BB(n,k+1)+\BB(n-1,k-1)$, proving the lemma.
\end{proof}

\begin{lemma}\label{lem:valid subs}
Fix $3\leq m\leq n$. There are $B(n-3,n-m)$ valid subdivision sequences $\mathbf{b}=(b_1,b_2,\dotsc,b_{n-m})$ that additionally satisfy $b_i\leq b_{i+1}$ for all $i$.
\end{lemma}

\begin{proof}
Let $\mathbf{b}=(b_1,b_2,\dotsc,b_{n-m})$ be a valid subdivision sequence with $b_i \le b_{i+1}$ for all $i$.  Define a new sequence $\mathbf{b}' = (b'_1,b'_2,\dotsc,b'_{n-3})$ by setting $b_i' = 1$ if $i \le m-3$ and $b_i'=b_{i-m+3}$ for $i>m-3$. We now have a nondecreasing sequence such that $b'_i \le i$ with an initial string of (at least) $m-3$ ones, so it satisfies the conditions of Lemma~\ref{ballot}.  One can easily check that this map is actually a bijection, and therefore it follows from the lemma that the number of such sequences is $B(n-3,n-m)$.
\end{proof}

\begin{proposition}\label{prop:descendantcount}
Fix $3\leq m\leq n$, and let $(\vd,\vr)$ be any smooth arithmetical structure on $D_m$. The number of arithmetical structures on $D_n$ that are descendants of $(\vd,\vr)$ is $B(n-3,n-m)$.
\end{proposition}

\begin{proof}
Note that Lemma~\ref{lem:ancestor} gives a bijection between arithmetical structures on $D_n$ that are descendants of a given arithmetical structure on $D_m$ and sequences $(b_1,b_2,\dotsc,b_{n-m})$ satisfying $1\leq b_i\leq m-3+i$ and $b_i\leq b_{i+1}$ for all $i$. By Lemma~\ref{lem:valid subs}, there are exactly $B(n-3,n-m)$ of these sequences.
\end{proof}

Let $\Arith(D_n)$ denote the set of arithmetical structures on $D_n$. We now count $\abs{\Arith(D_n)}$, the number of smooth arithmetical structures on $D_n$, in terms of $\abs{\SArith(D_m)}$, the number of smooth arithmetical structures on $D_m$, for all $m$ satisfying $4\leq m\leq n$.

\begin{theorem}\label{thm:smoothtooverall}
Let $n \ge 3$.  The number of arithmetical structures on $D_n$ is given by
\[
\abs{\Arith(D_n)}=2C(n-2)+\sum_{m=4}^nB(n-3,n-m)\abs{\SArith(D_m)}.
\]
\end{theorem}

\begin{proof}
We first count the number of arithmetical structures on $D_n$ with $d_x,d_y\geq2$. By Lemma~\ref{lem:uniquesmoothancestor}, each such arithmetical structure has a unique smooth ancestor on $D_m$ for some $m$ satisfying $3\leq m\leq n$. Proposition~\ref{prop:descendantcount} tells us that each smooth arithmetical structure on $D_m$ has $B(n-3,n-m)$ descendant arithmetical structures on $D_n$, and each of these has $d_x,d_y\geq2$. Thus the number of arithmetical structures on $D_n$ with $d_x,d_y\geq2$ is
\[
\sum_{m=3}^nB(n-3,n-m)\abs{\SArith(D_m)}.
\]

We next consider arithmetical structures on $D_n$ with $d_x=1$ or $d_y=1$. The set of arithmetical structures on $D_n$ with $d_x=1$ is in bijection with the set of arithmetical structures on the path graph with $n-1$ vertices by smoothing at the vertex $v_x$. Therefore, by~\cite{B17}, there are $C(n-2)$ such arithmetical structures. Similarly, there are $C(n-2)$ arithmetical structures on $D_n$ with $d_y=1$. The set of arithmetical structures on $D_n$ with $d_x=d_y=1$ is in bijection with the set of arithmetical structures on the path graph with $n-2$ vertices by smoothing at both $v_x$ and $v_y$, so there are $C(n-3)$ such structures. Thus there are $2C(n-2)-C(n-3)$ arithmetical structures on $D_n$ with $d_x=1$ or $d_y=1$.  Therefore the total number of arithmetical structures on $D_n$ is given by
\[
\abs{\Arith(D_n)}=2C(n-2)-C(n-3)+\sum_{m=3}^nB(n-3,n-m)\abs{\SArith(D_m)}.
\]

We simplify this expression by computing the term $B(n-3,n-3)\abs{\SArith(D_3)}$. By Proposition~\ref{prop:centraldis1}, a smooth arithmetical structure on $D_3$ must have $d_0=1$. Therefore $r_0=r_x+r_y=\frac{r_0}{d_x}+\frac{r_0}{d_y}$. Since $d_x,d_y\geq2$, this implies $d_x=d_y=2$. Hence there is a unique smooth arithmetical structure on $D_3$, namely that with $\vd=(2,2,1)$ and $\vr=(1,1,2)$, so $\abs{\SArith(D_3)}=1$. Also, $B(n-3,n-3)=C(n-3)$, so therefore $B(n-3,n-3)\abs{\SArith(D_3)}=C(n-3)$. Hence the above expression simplifies to give
\[
\abs{\Arith(D_n)}=2C(n-2)+\sum_{m=4}^nB(n-3,n-m)\abs{\SArith(D_m)}.\qedhere
\]
\end{proof}

When $n=3$, the sum in Theorem~\ref{thm:smoothtooverall} is empty, and therefore $\abs{\Arith(D_3)}=2C(3-2)=2$.

Note that Theorem~\ref{thm:smoothtooverall} shows $\abs{\Arith(D_n)}$ grows at least as fast as $2C(n-2)$. In Section~\ref{sec:bounds}, after establishing an upper bound on $\abs{\SArith(D_m)}$, we will obtain an upper bound on $\abs{\Arith(D_n)}$ that is also a multiple of $C(n-2)$, thus showing that $\abs{\Arith(D_n)}$ grows at the same rate as $C(n-2)$.

In summary, this section has reduced the problem of counting arithmetical structures on $D_n$ to that of counting smooth arithmetical structures on $D_m$ for all $m$ satisfying $4\leq m\leq n$. We address the question of counting smooth arithmetical structures on bidents in the next section.

\section{Counting smooth arithmetical structures}\label{sec:reduction}

By Theorem~\ref{thm:smoothtooverall}, in order to enumerate arithmetical structures on $D_n$, it is enough to restrict attention to smooth arithmetical structures on $D_n$ and smaller bidents. In this section, we determine the number of smooth arithmetical structures on $D_n$ in terms of a number-theoretic function $F$, defined in this section. We use these results in Section~\ref{sec:bounds} to understand the growth rates of the number of smooth arithmetical structures and the number of arithmetical structures on $D_n$ as $n$ increases.

In this section and the following, it will be convenient to use a scalar multiple of the primitive vector $\mathbf{r}$. Specifically, we instead work with $\rab=\frac{r_0}{r_xr_y}\mathbf{r}$. Since $r_x$ and $r_y$ both divide $r_0$ and $\gcd(r_x, r_y) =1$ by Lemma~\ref{lem:gcd}$(a)$, the vector $\rab$ is comprised of positive integer entries. We also note that $\rab$ is exactly the scalar multiple of $\vr$ such that $\r_x\r_y=\r_0$.

\subsection{Determining structures from \texorpdfstring{$\bm{\r_x}$}{r\_x} and \texorpdfstring{$\bm{\r_y}$}{r\_y}}

We first observe that the values of $\r_x$ and $\r_y$ uniquely determine a smooth arithmetical structure. The following proposition is an immediate corollary of Proposition~\ref{prop:abcuniqueness}, taking the triple from that proposition to be $(a',b',c')=(a, b, ab)/\gcd(a,b)$ and rescaling the $\vr$-vector.

\begin{proposition}\label{prop:abuniqueness}
For every pair of integers $a,b\geq2$, there is a unique $n\geq3$ such that there is a smooth arithmetical structure on $D_n$ with $\r_x=a$ and $\r_y=b$. Moreover, this smooth arithmetical structure on $D_n$ with $\r_x=a$ and $\r_y=b$ is unique.
\end{proposition}

We will obtain a more precise version of Proposition~\ref{prop:abuniqueness} in Theorem~\ref{T:length}. In order to do this, we first define a function $F\colon\mathbb{Z}_{>0} \times \mathbb{Z}_{\geq 0} \to \mathbb{Z}_{>0}$ as follows. Given a positive integer $x_1$ and a nonnegative integer $x_2$, we define a sequence $\{x_i\}$ by setting $x_{i+1}$ to be the least nonnegative residue of $-x_{i-1}$ modulo $x_i$, as long as $x_i>0$. Note that this means $x_{i+1}$ is the unique integer with $0\leq x_{i+1}<x_i$ and $x_i\mid x_{i-1}+x_{i+1}$. Let $k$ be the largest value of $i$ for which $x_i$ is nonzero (i.e.\ such that $x_k\mid x_{k-1}$ with $k\geq 2$). Define $F(x_1,x_2)=k$, the number of  positive terms in the sequence $\{x_i\}$. Note that, for any $x>0$, we have that $F(x,0)=1$, since there is only one positive term in the sequence.

As an example, suppose we want to compute $F(17,12)$. Then we take $x_1=17$ and $x_2=12$. The value of $x_3$ will be the least residue of $-17$ modulo 12. So $x_3= 7$. Notice that $7$ is also the smallest positive value of $x_3$ for which $12 \mid (17+x_3)$. We similarly compute $x_4=2$ and $x_5 = 1$. Since we must then have $x_6=0$, we determine that $F(17,12)=5$.

Comparing the definition of $F$ with the construction in the proof of Proposition~\ref{prop:abcuniqueness}, we see that, if we have a smooth arithmetical structure with $\r_0=x_1$ and $\r_1=x_2$, we must then have $\r_i=x_{i+1}$ for all $i$ satisfying $0\leq i\leq \ell$. This means that, if we have a smooth arithmetical structure on $D_n$ with $\rab= (\r_x, \r_y, \r_0, \r_1, \dotsc, \r_\ell)$, we then have $F(\r_0,\r_1)=\ell+1=n-2$, and hence that $n=F(\r_0,\r_1)+2$. It also follows from Lemma~\ref{lem:gcd}$(b)$ that, if $x_k$ is the last positive term in the sequence $\{x_i\}$, we must have $x_k=\gcd(x_1,x_2)$.

The function $F$ will be useful both in Theorem~\ref{T:length} below and in Section~\ref{sec:bounds}, where we will establish a relationship between $F$ and the Euclidean algorithm. We begin with a lemma.

\begin{lemma}\label{lem:Fbasics}
Let $x$ be a positive integer, and let $y$ and $k$ be nonnegative integers. We have the following:
\begin{enumerate}[ $(a)$]
\item $F(x,y)=F(x+ky,y)$,
\item $F(x,kx+y)=F(x,y)+k$,
\item $F(ax,ay)=F(x,y)$,
\item $F(x,x-1)=x$,
\item $F(x,y)\leq y+1$,
\item $F(x,y) \le \frac{x+1}{2}$ if $1 \le y \le x-2$.
\end{enumerate}
\end{lemma}

\begin{proof}
Part $(a)$ follows from the fact that $-x\equiv-(x+ky)\pmod y$.
For $(b)$, first consider the case when $k=1$. Note that $-x\equiv y\pmod{x+y}$, so therefore $F(x,x+y)=F(x+y,y)+1$. Using~$(a)$ then gives that $F(x,x+y)=F(x,y)+1$. The general statement of $(b)$ follows by induction. Parts~$(c)$, $(d)$, and $(e)$ are immediate from the definition.

For part $(f)$, if $y < \frac{x}{2}$ we can use $(e)$ to get $F(x,y)\le y+1 < \frac{x}{2}+1$.  On the other hand, if $y\geq \frac{x}{2}$ we set $k=x-y$, noting that $2 \le k \leq \frac{x}{2}$. We then choose integers $q$ and $r$ so that $x=qk+r$ and $0 \le r < k$.  If $r=0$, we use parts $(c)$ and $(d)$ to compute that
\[
F(x,y)=F(qk,(q-1)k)=F(q,q-1)=q = \frac{x}{k} \le \frac{x}{2}.
\]
If $r>0$, we then have $q<\frac{x}{k}$, and so we can use parts $(a)$ and $(b)$ to compute that
\[
F(x,y) = F(qk+r,(q-1)k+r)=F(k,(q-1)k+r) = F(k,r)+(q-1).
\]
Applying $(e)$, we have that $F(x,y) \le q+r<\frac{x}{k}+k-1$. Because $k\leq\frac{x}{2}$, we have that $2k(k-2)\leq x(k-2)$. It follows that $2(x+k(k-1))\leq k(x+2)$, from which we deduce that $\frac{x}{k}+k-1\leq \frac{x}{2}+1$.  Therefore $F(x,y)<\frac{x}{2}+1$ in all cases. The statement follows.
\end{proof}

\begin{theorem}\label{T:length}
Given a pair of integers $a,b\geq2$, write $a+b=tb^2+\epsilon$, where $t$ and $\epsilon$ are integers satisfying $t\geq0$ and $0 \le \epsilon \le b^2-1$. There is a unique smooth arithmetical structure on $D_n$ with $\r_x=a$ and $\r_y=b$ if $n = F(b^2,\epsilon)+t+ \floor*{\frac{ab}{a+b}}$ and no such structure for all other choices of $n$.
\end{theorem}

\begin{proof}
Proposition~\ref{prop:abuniqueness} states that there is a unique smooth arithmetical structure with $\r_x=a$ and $\r_y=b$ on $D_n$ for one $n$ and no such structure for any other $n$. Since $\r_1=ab-a-b$, this unique smooth arithmetical structure occurs when $n=F(\r_0,\r_1)+2=F(ab,ab-a-b)+2$.

Let $c=a+b$ and $k=
\floor*{\frac{ab}{a+b}}$. Note that our hypotheses imply that $k \ge 1$ and that we can write $ab = kc + d$ for some $0\leq d <c$. Then
\[
F(ab,ab-a-b)=F((ab-c)+c, ab-c)=F(c,ab-c),
\]
by Lemma~\ref{lem:Fbasics}$(a)$. Since $ab = kc + d$, we have $ab-c=(k-1)c + d$, and by Lemma~\ref{lem:Fbasics}$(b)$ we can write
\[
F(c, ab-c) = F(c, (k-1)c + d)= F(c,d)+k-1.
\]
This implies that $n=F(c,d) + k+1$. Now let us compare this to the computation of $F(b^2,c)$. Since $-b^2 \equiv ab\pmod c$ and $ab=kc+d$, we have that $F(b^2,c)=F(c, d)+1$. It therefore follows that $n=F(b^2,c)+k$.

We also have that $F(b^2,c)=F(b^2,tb^2+\epsilon)=F(b^2,\epsilon)+t$ by Lemma~\ref{lem:Fbasics}$(b)$. Therefore $n=F(b^2,\epsilon)+t+k$, as desired.
\end{proof}

We remark that, although the expression for $n$ in the above theorem does not appear to be symmetric in $a$ and $b$, it in fact is. As in the last paragraph of the proof, $F(b^2,\epsilon)+t=F(b^2,a+b)$, and, since $a^2\equiv b^2\pmod{a+b}$, Lemma~\ref{lem:Fbasics}$(a)$ implies that $F(b^2,a+b)=F(a^2,a+b)$.

\begin{corollary}\label{Cor:length}
Let $a,b\geq2$ be integers, and write $a+b=tb^2+\epsilon$, where $t$ and $\epsilon$ are integers satisfying $t\geq0$ and $0 \le \epsilon \le b^2-1$. If $t\geq1$, then the smooth arithmetical structure with $\r_x=a$ and $\r_y=b$ on $D_n$ occurs when $n = F(b^2,\epsilon)+t+b-1$.
\end{corollary}

\begin{proof}
Note that for any positive integers $a$ and $b$ we have $(a+b)b = ab+b^2 > ab$. Since $t\geq1$, we have $a+b \ge b^2$, and therefore we can compute that $(a+b)(b-1)=ab+b^2-a-b \le ab$. Hence, $b-1\leq\frac{ab}{a+b}< b$ and thus $\floor*{\frac{ab}{a+b}} = b-1$.  The result then follows from Theorem~\ref{T:length}.
\end{proof}

Suppose we are in the situation where $a\geq b^2-b$. In this case, increasing $a$ by $b^2$, which leaves~$\epsilon$ unchanged and increases $t$ by $1$, has the effect of increasing $n$ by exactly $1$. Therefore, for each $b\geq 2$ and $\epsilon$ in the range $0\leq\epsilon\leq b^2-1$, we get exactly one smooth arithmetical structure on $D_n$ for all $n\geq F(b^2,\epsilon)+b$. This leads to the following result.

\begin{theorem}\label{T1}
For any $b\geq 2$ and any $n \ge b^2+b$, there are exactly $b^2$ smooth arithmetical structures on $D_n$ with $\r_y=b$.
\end{theorem}

\begin{proof}
We note that if $0\leq\epsilon \le b^2-1$ then $F(b^2,\epsilon) \le b^2$ and therefore the above comments show that there will be at least $b^2$ smooth arithmetical structures on $D_n$ for all $n\ge b^2+b$.  Moreover, if $a < b^2-b$ then $t=0$ and $\floor*{\frac{ab}{a+b}} < b-1$, so $F(b^2,\epsilon)+t+ \floor*{\frac{ab}{a+b}}<b^2-b$, meaning none of these structures can occur on $D_n$ with $n\geq b^2-b$. This implies the theorem.
\end{proof}

\subsection{Bounding entries and counting}

Before using Theorem~\ref{T:length} to count smooth arithmetical structures, we first prove that, for a fixed value of $n$, we cannot have $\r_x$ and $\r_y$ both be too large. In this subsection, we take $a=\max\{\r_x,\r_y\}$ and $b=\min\{\r_x,\r_y\}$, where $\rab$ is associated to some smooth arithmetical structure. Therefore we always have $b\leq a$. Notice that any pair $(a,b)$ with $b<a$ gives rise to two smooth arithmetical structures on $D_n$ for some $n$, one with $\r_x=a$ and $\r_y=b$ and another with $\r_x=b$ and $\r_y=a$.

\begin{proposition}\label{thm:2n-4}
Let $n\geq 3$. For every smooth arithmetical structure on $D_n$, we must have $2\leq b\leq 2n-4$.
\end{proposition}
\begin{proof}
Let $n\geq 3$ and let $\rab$ be associated to a smooth arithmetical structure on $D_n$. First note that if $b=1$ then $\r_0=ab=a$, in which case $d_x=1$ or $d_y=1$. However the definition of smooth requires $d_x, d_y\geq2$ and thus we must have $b\geq 2$. To show that $b\leq 2n-4$, we proceed by contradiction, showing that $a\geq b>2n-4$ leads to $\floor*{\frac{ab}{a+b}} + t +F(b^2,\epsilon)>n$.

It is straightforward to see that $\frac{ab}{a+b}$ is increasing in both $a$ and $b$.  Thus, if $a\geq b \ge 2n-2$, we have $\frac{ab}{a+b} \geq \frac{(2n-2)^2}{2(2n-2)} = n-1$.  The fact that $F(b^2,\epsilon)+t \ge 2$ then leads to a contradiction.

If $b=2n-3$ and $a \ge 2n$, we can compute that
\[
\frac{ab}{a+b} \ge \frac {(2n-3)(2n)}{4n-3} = \frac{4n^2-6n}{4n-3} = n-1 + \frac{n-3}{4n-3}>n-1,
\]
which similarly leads to a contradiction.

Finally, suppose that $b=2n-3$ and $2n-3 \le a \le 2n-1$. In this case, $\floor*{\frac{ab}{a+b}} = n-2$.  If $a$ is odd, then $a+b$ will be even.  In particular, this means $a+b$ cannot be a divisor of $b^2$ since $b^2$ is odd, so $F(b^2,a+b)>2$. Therefore $\floor*{\frac{ab}{a+b}} + t +F(b^2,\epsilon)>n$,  a contradiction. On the other hand, if $a=2n-2$, we have that $a+b=4n-5$.  Since $\gcd(4n-5,4n-6)=1$, it follows that $4n-5$ cannot be a divisor of $b^2=(2n-3)^2$.  This implies that $F(b^2,a+b)>2$, again leading to a contradiction.
\end{proof}

The preceding results are enough to give a finite procedure for finding all smooth arithmetical structures on $D_n$ for any fixed $n$: For all $b$, $t$, and $\epsilon$ in the ranges $2\leq b\leq 2n-4$, $0\leq \epsilon\leq b^2-1$, and $0\leq t\leq n-2$, check whether $F(b^2,\epsilon)+t+ \floor*{\frac{b(tb^2-b+\epsilon)}{tb^2+\epsilon}}=n$. If equality does hold, we count two smooth arithmetical structures for every such triple $(b,t,\epsilon)$ with $a=tb^2+\epsilon-b>b$ and one smooth arithmetical structure for every such triple $(b,t,\epsilon)$ with $a=tb^2+\epsilon-b=b$. We make some additional observations that make this algorithm more efficient and that will be helpful in establishing bounds on $\abs{\SArith(D_n)}$ in Section~\ref{sec:bounds}.

\begin{lemma}\label{lem:bebound}
Fix $n\geq 3$. For $b$ satisfying $2\leq b\leq 2n-4$ and $\epsilon$ satisfying $0\leq \epsilon\leq b^2-1$, there are at most two smooth arithmetical structures on $D_n$ corresponding to $(b,\epsilon)$.
\end{lemma}

\begin{proof}
It is a straightforward exercise to check that $\frac{(tb^2+\epsilon-b)b}{tb^2+\epsilon}$ is an increasing function of $t$; hence $F(b^2,\epsilon)+t+\floor*{\frac{(tb^2+\epsilon-b)b}{tb^2+\epsilon}}$ is an increasing function of $t$. Therefore, for fixed $n$, $b$, and $\epsilon$, there is at most one value of $t$ for which $F(b^2,\epsilon)+t+\floor*{\frac{(tb^2+\epsilon-b)b}{tb^2+\epsilon}}=n$. A triple $(b,t,\epsilon)$ thus gives rise to two smooth arithmetical structures on $D_n$ if $a=tb^2+\epsilon-b$ is greater than $b$: one with $\r_x=a$ and $\r_y=b$ and the other with $\r_x=b$ and $\r_y=a$. If $a=b$, there is one arithmetical structure with $\r_x=\r_y=b$.
\end{proof}

Recall from the proof of Corollary~\ref{Cor:length} that, when $t\geq1$, we have $\floor*{\frac{(tb^2+\epsilon-b)b}{tb^2+\epsilon}}=b-1$. We thus have the following specific possibilities for fixed $n\geq4$, $b$, and $\epsilon$:

\begin{itemize}
\item $F(b^2,\epsilon)+b-1<n$, in which case we can set $t=n-F(b^2,\epsilon)-b+1\geq1$ and get a pair of smooth arithmetical structures on $D_n$ corresponding to $(b,\epsilon)$;
\item $F(b^2,\epsilon)+b-1\geq n$ (meaning the only possibility is $t=0$) with $F(b^2,\epsilon)+\floor*{\frac{(\epsilon-b)b}{\epsilon}}<n$ or $\epsilon<2b$, in which case there is no smooth arithmetical structure on $D_n$ corresponding to $(b,\epsilon)$;
\item $F(b^2,\epsilon)+b-1\geq n$ with $F(b^2,\epsilon)+\floor*{\frac{(\epsilon-b)b}{\epsilon}}=n$ and $\epsilon>2b$, in which case there are two smooth arithmetical structures on $D_n$ corresponding to $(b,\epsilon)$;
\item $F(b^2,\epsilon)+b-1\geq n$ with $F(b^2,\epsilon)+\floor*{\frac{(\epsilon-b)b}{\epsilon}}=n$ and $\epsilon=2b$, in which case there is one smooth arithmetical structure on $D_n$ corresponding to $(b,\epsilon)$; or
\item $F(b^2,\epsilon)+b-1\geq n$ with $F(b^2,\epsilon)+\floor*{\frac{(\epsilon-b)b}{\epsilon}}>n$ and $\epsilon\geq 2b$, in which case there is no smooth arithmetical structure on $D_n$ corresponding to $(b,\epsilon)$.
\end{itemize}

Therefore we find the number of smooth arithmetical structures on $D_n$ for $n\geq4$ by determining which of the above cases we are in for all values of $b$ and $\epsilon$ in the ranges $2\leq b\leq 2n-4$ and $0\leq\epsilon\leq b^2-1$. We have implemented this algorithm for $n$ in the range $4\leq n\leq 43$; the results are shown in Table~\ref{tab:counts-} and also illustrated in Figure~\ref{fig:cubicgrowth}. We then use Theorem~\ref{thm:smoothtooverall} to find the total number of arithmetical structures on $D_n$; these results also appear in Table~\ref{tab:counts-}. We remark that this algorithm is efficient in practice; the data in Table~\ref{tab:counts-} were generated in less than one minute using \texttt{SageMath}~\cite{sage} on a standard desktop computer.

\begin{table}
\centering
\begin{tabular}{|c|c|c||c|c|c|}
\hline
\multicolumn{1}{|c|}{$n$}& \multicolumn{1}{c|}{$\abs{\SArith(D_n)}$}& \multicolumn{1}{c||}{$\abs{\Arith(D_n)}$}&\multicolumn{1}{c|}{$n$}& \multicolumn{1}{c|}{$\abs{\SArith(D_n)}$}& \multicolumn{1}{c|}{$\abs{\Arith(D_n)}$}\\
\hline
\phantom{1}4&\phantom{2,9}10&\phantom{716,420,218,8}14& 24&\phantom{2}3,806&\phantom{339,028,157,11}2,711,456,910,222\\
\phantom{1}5&\phantom{2,9}16&\phantom{716,420,218,8}46& 25&\phantom{2}3,958&\phantom{339,028,157,1}10,281,958,081,812\\
\phantom{1}6&\phantom{2,9}50&\phantom{716,420,218,}176& 26&\phantom{2}5,022&\phantom{339,028,157,1}39,059,990,775,594\\
\phantom{1}7&\phantom{2,9}52&\phantom{716,420,218,}620& 27&\phantom{2}5,054&\phantom{339,028,157,}148,635,185,291,644\\
\phantom{1}8&\phantom{2,}126&\phantom{716,420,21}2,218& 28&\phantom{2}6,236&\phantom{339,028,157,}566,498,545,019,834\\
\phantom{1}9&\phantom{2,}124&\phantom{716,420,21}7,938& 29&\phantom{2}6,380&\phantom{339,028,15}2,162,330,791,492,290\\
10&\phantom{2,}250&\phantom{716,420,2}28,572& 30&\phantom{2}7,946&\phantom{339,028,15}8,265,205,867,169,156\\
11&\phantom{2,}244&\phantom{716,420,}103,384& 31&\phantom{2}8,106&\phantom{339,028,1}31,634,330,508,005,370\\
12&\phantom{2,}434&\phantom{716,420,}376,056& 32&\phantom{2}9,612&\phantom{339,028,}121,228,606,496,811,950\\
13&\phantom{2,}432&\phantom{716,42}1,374,680& 33&10,060&\phantom{339,028,}465,118,574,235,674,538\\
14&\phantom{2,}690&\phantom{716,42}5,048,348& 34&11,744&\phantom{339,02}1,786,517,442,487,495,664\\
15&\phantom{2,}710&\phantom{716,4}18,618,290& 35&12,104&\phantom{339,02}6,869,273,566,377,014,478\\
16&1,032&\phantom{716,4}68,932,582& 36&14,320&\phantom{339,0}26,439,373,973,414,097,184\\
17&1,066&\phantom{716,}256,133,188& 37&14,736&\phantom{339,}101,860,743,777,136,381,978\\
18&1,552&\phantom{716,}954,856,744& 38&17,006&\phantom{339,}392,787,703,022,696,559,172\\
19&1,576&\phantom{71}3,570,492,960& 39&17,560&\phantom{33}1,515,952,946,666,164,348,660\\
20&2,114&\phantom{7}13,388,550,056& 40&20,050&\phantom{33}5,855,622,326,076,661,242,226\\
21&2,190&\phantom{7}50,334,109,160& 41&20,586&\phantom{3}22,636,211,612,489,393,913,770\\
22&2,874&189,684,561,610& 42&23,824&\phantom{3}87,571,480,303,245,046,251,032\\
23&2,946&716,420,218,810& 43&24,310&339,028,157,112,678,873,881,416\\
\hline
\end{tabular}
\caption{The number of smooth arithmetical structures and the total number of arithmetical structures on $D_n$ for $n$ in the range $4\leq n\leq 43$.\label{tab:counts-}}
\end{table}

We end this section by observing that there appears to be a parity issue in the data in Table~\ref{tab:counts-}. Specifically, for at least $n\leq200$, we have that $\abs{\SArith(D_{n})}-\abs{\SArith(D_{n-1})}$ is larger than $\abs{\SArith(D_{n+1})}-\abs{\SArith(D_{n})}$ when $n$ is even and smaller when $n$ is odd. At this time, we do not have a good explanation of this parity issue, but it appears to be due to smooth arithmetical structures obtained from pairs $(a,b)$ for which $a+b<\min\{a^2, b^2\}$.

\section{Bounds}\label{sec:bounds}
In this section, we show how to bound the number of smooth arithmetical structures on $D_n$ and the total number of arithmetical structures on $D_n$. We first show that $\abs{\SArith(D_n)}$ grows cubically in the sense that it is bounded above and below by cubic functions of $n$.
\begin{theorem}\label{thm:Xnbounds}
Let $\abs{\SArith(D_n)}$ be the number of smooth arithmetical structures on $D_n$. Then
\[
\frac{1}{24}(n^3-3n^2-n-45)\leq\abs{\SArith(D_n)}<\frac{2}{3}n^3-2n^2 +\frac{16}{3}n-6+(2n^2-4n+2)\log(n-3).
\]
\end{theorem}

\begin{figure}
\centering
\begin{tabular}{cc}
\subcaptionbox{using $4\leq n\leq 50$}{
\resizebox{.465\textwidth}{!}{\includegraphics{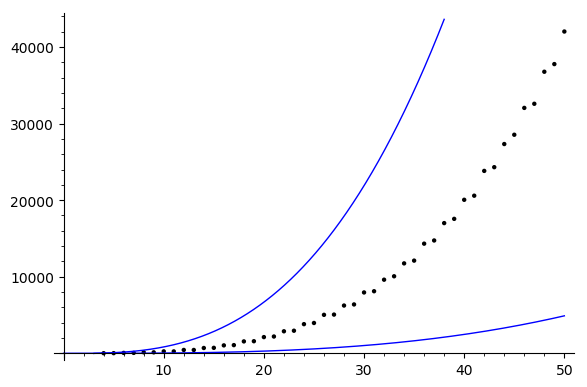}}}
&\subcaptionbox{using $4\leq n\leq 200$}{
\resizebox{.465\textwidth}{!}{\includegraphics{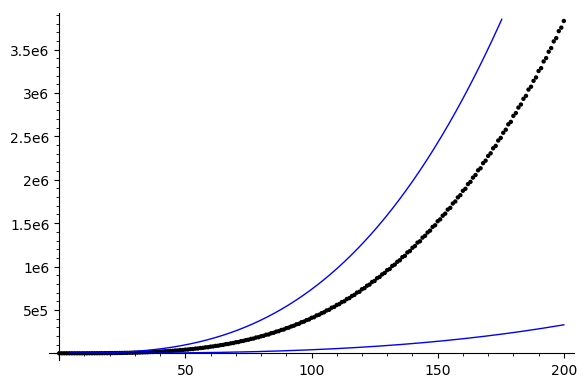}}}
\end{tabular}
\caption{Graphs of $(n,\abs{\SArith(D_n)})$ for $n$ in the range $4\leq n\leq 50$ and in the range $4\leq n\leq 200$ together with the upper and lower bounds given by Theorem~\ref{thm:Xnbounds}.\label{fig:cubicgrowth}\label{fig:sbound}}
\end{figure}

This theorem is illustrated in Figure~\ref{fig:cubicgrowth}. We note that the data suggest $\abs{\SArith(D_n)}$ is well approximated by a cubic polynomial with leading coefficient approximately $0.6$. Therefore we believe the upper bound in Theorem~\ref{thm:Xnbounds} is quite good. On the other hand, the lower bound in this theorem, while of the right order, seems quite far from being optimal.

In Theorem~\ref{Thm:overallcount}, we use the above theorem to show that $\abs{\Arith(D_n)}$ grows at the same rate as the Catalan numbers. As in Section~\ref{sec:reduction}, we work with the vector $\rab=\frac{r_0}{r_xr_y}\mathbf{r}$ and define $a=\max\{\r_x,\r_y\}$ and $b=\min\{\r_x,\r_y\}$.  We again define integers $t$ and $\epsilon$ so that $a+b=tb^2+\epsilon$ with $t\geq 0$ and $0\leq \epsilon \leq b^2-1$.

\subsection{Upper bound on number of smooth arithmetical structures}

We first note that the results of Section~\ref{sec:reduction} immediately give an upper bound on $\abs{\SArith(D_n)}$. Proposition~\ref{thm:2n-4}  shows that $2\leq b\leq 2n-4$, and Lemma~\ref{lem:bebound} shows that, for each $\epsilon$ satisfying $0\leq \epsilon \leq b^2-1$, there are at most two smooth arithmetical structures on $D_n$ corresponding to the pair $(b,\epsilon)$. Therefore there are at most $2b^2$ smooth arithmetical structures on $D_n$ corresponding to a given $b$, and we have that
\[
\abs{\SArith(D_n)}\leq \sum_{b=2}^{2n-4} 2b^2 = \frac{2}{3}(8n^3-42n^2+73n-45).
\]

This bound is not sharp, both because there are sometimes no structures corresponding to a pair $(b,\epsilon)$ and because this bound double counts structures where $\r_x$ and $\r_y$ are both at most $2n-4$. In the following proposition, we improve this bound by treating the cases $b<n$ and $b\geq n$ separately.

\begin{proposition}\label{prop:upperbound}
For all $n\geq 4$, the number of smooth arithmetical structures on $D_n$  is bounded above as follows:
\[
\abs{\SArith(D_n)} < \frac{2}{3}n^3-2n^2 +\frac{16}{3}n-6+(2n^2-4n+2)\log(n-3).
\]
\end{proposition}
\begin{proof}

For each $b$ in the range $2\leq b\leq n-1$, we use the same approach as above to say that there are at most $2b^2$ smooth arithmetical structures. If $a=b$, it is straightforward to see that $n= \ceil*{\frac{b}{2}} +2$, so there are exactly two such smooth arithmetical structures for each $n$.  To bound the number of possibilities for $b$ satisfying $n\leq b\leq 2n-4$ and with $a>b$, we proceed as follows.

Recall that in Theorem~\ref{T:length} we let $\floor*{\frac{ab}{a+b}} + t + F(b^2,\epsilon)=n$. Note that $t+F(b^2,\epsilon) \ge 2$ since otherwise we would have $t=\epsilon=0$ which would imply $a+b=0$.  Therefore $a<\frac{b(n-1)}{b-n+1}$. With the restriction $a>b$, there are thus at most $\ceil*{\frac{b(n-1)}{b-n+1}}-b-1$ such values of $a$ corresponding to~$b$. Each such pair $(b,a)$ gives two smooth arithmetical structures on $D_n$, so there are therefore at most $2\left(\ceil*{\frac{b(n-1)}{b-n+1}}-b-1\right)$ smooth arithmetical structures corresponding to $b$ in this case.

We thus obtain the following upper bound:
\begin{align*}
\abs{\SArith(D_n)} &\leq  \sum_{b=2}^{n-1} 2b^2 + 2 + \sum_{b=n}^{2n-4} 2\left(\ceil*{\frac{b(n-1)}{b-n+1}} -b - 1\right) \\
& \leq  \frac{1}{3}(2 n^3 - 3 n^2 + n - 6) + 2 - (3n^2-13n+12) + 2 \sum_{b=n}^{2n-4} \frac{b(n-1)}{b-n+1}\\
& =  \frac{2}{3}n^3 -4n^2+\frac{40}{3}n-12 + 2 \sum_{c=1}^{n-3}\frac{(c+n-1)(n-1)}{c}\\
& =  \frac{2}{3}n^3 -4n^2+\frac{40}{3}n-12 + 2 \sum_{c=1}^{n-3} (n-1) + 2 \sum_{c=1}^{n-3}\frac{(n-1)^2}{c}\\
& =  \frac{2}{3}n^3 -4n^2+\frac{40}{3}n-12 + 2 (n^2-4n+3) + 2(n-1)^2 H_{n-3}\\
& <  \frac{2}{3}n^3-2n^2 +\frac{16}{3}n-6+(2n^2-4n+2)\log(n-3).
\end{align*}
Note that the above computation makes use of the substitution $c=b-n+1$ as well as the fact that the partial sums $H_n$ of the harmonic series satisfy the bound $H_n <1+\log(n)$.
\end{proof}

\subsection{Lower bound on number of smooth arithmetical structures}
We now turn to the lower bound of Theorem~\ref{thm:Xnbounds}. The general strategy is to show that there are sufficiently many values of $b$ and $\epsilon$ for which $F(b^2,\epsilon)+b\leq n$. For such values of $b$ and $\epsilon$, we can set $t=n-F(b^2,\epsilon)-b+1\geq1$ in accordance with Corollary~\ref{Cor:length}, thus obtaining a pair of smooth arithmetical structures on $D_n$ corresponding to $(b,\epsilon)$. With this aim, we first study $F(\beta,\epsilon)$ for an arbitrary pair $(\beta,\epsilon)$. We will later set $\beta=b^2$ when applying the results to the above setting.

We begin by establishing a connection between $F(\beta,\epsilon)$ and certain quotients that appear in the Euclidean algorithm. Let $\beta,\epsilon\in \mathbb{Z}_{>0}$ with $\epsilon<\beta$, and denote by $\q_{\beta,\epsilon}=(q_1, q_2,\dotsc, q_k)$ the vector consisting of the quotients that appear in the Euclidean algorithm when performed on $(\beta,\epsilon)$. Specifically, we have
\begin{align}
\beta       & = q_1\epsilon+r_1 \nonumber                 \\
\epsilon       & = q_2r_1+r_2 \nonumber               \\
r_1     & = q_3r_2+r_3 \nonumber \\
&\vdotswithin{=}\label{eq:gcdEuclidean}\\
r_{k-3} & = q_{k-1}r_{k-2}+r_{k-1}\nonumber    \\
r_{k-2} & = q_kr_{k-1}+0.\nonumber
\end{align}

As an example, suppose $\beta=36$ and $\epsilon=23$. The Euclidean algorithm then gives
\[
36=1\cdot 23+13, \qquad 23=1\cdot 13+10, \qquad 13=1\cdot 10+3, \qquad 10=3\cdot 3+1, \qquad 3=3\cdot 1+0.
\]
We thus have $\q_{36,23}=(q_1,q_2,q_3,q_4,q_5)=(1,1,1,3,3)$. Lemma~\ref{lemma:Fintermsofq} will establish that $F(36,23)=q_2+q_4+2=6$ and $F(36,36-23)=F(36,13)=q_1+q_3+q_5=5$. Indeed, these results are true in this example; the sequence $(36,23,10,7,4,1)$ shows that $F(36,23)=6$, and the sequence $(36,13,3,2,1)$ shows that $F(36,13)=5$. We first prove the following general lemma.

\begin{lemma}\label{lem:FEuclidean}
Suppose $x=q_1y+r_1$ and $y=q_2r_1+r_2$, where $x$, $y$, and $r_1$ are positive integers and $q_1$, $q_2$, and $r_2$ are nonnegative integers. Then $F(x,y)=q_2+F(r_1,r_2)$.
\end{lemma}

\begin{proof}
Using parts $(a)$ and $(b)$ of Lemma~\ref{lem:Fbasics}, we have that
\[
F(x,y)=F(q_1y+r_1,y)=F(r_1,y)=F(r_1,q_2r_1+r_2)=q_2+F(r_1,r_2).\qedhere
\]
\end{proof}

Before stating Lemma~\ref{lemma:Fintermsofq}, we define the following notation. For any vector $\q=(q_1, q_2, \dotsc, q_k)$, let
\[
S_{\q}^o = \sum_{\substack{i \text{ odd}\\1\leq i \leq k}} q_i \qquad \text{ and } \qquad S_{\q}^e = \sum_{\substack{i \text{ even}\\1\leq i \leq k}} q_i.
\]

\begin{lemma}\label{lemma:Fintermsofq}
Let $\beta,\epsilon\in \mathbb{Z}_{>0}$ with $\epsilon<\beta$, $\q=\q_{\beta,\epsilon}$, and $k=\abs{\q}$. Then
\[
F(\beta,\epsilon)=\begin{cases}
\displaystyle S_{\q}^e+1 & \text{ if $k$ is even} \\
\displaystyle S_{\q}^e+2 & \text{ if $k$ is odd,}  \end{cases}
\]
and
\[
F(\beta,\beta-\epsilon)=\begin{cases}
\displaystyle S_{\q}^o+1 & \text{ if $k$ is even} \\
\displaystyle S_{\q}^o & \text{ if $k$ is odd.}  \end{cases}
\]
\end{lemma}

\begin{proof}
First consider $F(\beta,\epsilon)$. If $k$ is even, we apply Lemma~\ref{lem:FEuclidean} as often as possible to get that $F(\beta,\epsilon)=S_{\q}^e+F(r_{k-1},0)=S_{\q}^e+1$. If $k$ is odd, we apply Lemma~\ref{lem:FEuclidean} as often as possible to get $F(\beta,\epsilon)=S_{\q}^e+F(r_{k-2},r_{k-1})=S_{\q}^e+2$.

Now let $\q'=\q_{\beta+\epsilon,\beta}$. Observe that $\q'=(1,q_1,q_2,\dotsc,q_k)$. If $k$ is even, $F(\beta,\beta-\epsilon)=F(\beta+\epsilon,\beta)-1=S_{\q'}^e+F(r_{k-2},r_{k-1})-1=S_{\q}^o+1$. If $k$ is odd, $F(\beta,\beta-\epsilon)=F(\beta+\epsilon,\beta)-1=S_{\q'}^e+F(r_{k-1},0)-1=S_{\q}^o$.
\end{proof}

We next make some observations about the sums $S_{\q}^o$ and $S_{\q}^e$ that appear in the above lemma. To do this, we define
\[
M^{(k)}_q=\prod_{i=1}^{k} \begin{pmatrix} q_i&1 \\ 1&0 \end{pmatrix},
\]
and then define $A_k$ to be the upper left entry of $M_q^{(k)}$, i.e.\ $A_k\coloneqq(M^{(k)}_\q)_{1,1}$. The key properties of the $A_k$ are that $A_1=q_1$, $A_2=q_1q_2+1$, and $A_k= q_k(A_{k-1})+A_{k-2}$ for all $k\geq2$.

\begin{lemma}\label{lem:Akbounds}
Let $\q^k=(q_1,q_2,\dotsc,q_k)\in \mathbb{Z}_{>0}^k$ for $k\geq2$. Then
\begin{enumerate}[ $(a)$]
\item if $k$ is even, $A_k\geq S_{\q^{k-1}}^o\cdot S_{\q^k}^e+1$, and
\item if $k$ is odd, $A_k\geq S_{\q^{k}}^o+q_k(S_{\q^{k-2}}^o\cdot S_{\q^{k-1}}^e)$.
\end{enumerate}
\end{lemma}

\begin{proof}
We proceed by induction with $k=2$ and $k=3$ as base cases. When $k=2$, we have that
\[
S_{\q^1}^o \cdot S_{\q^2}^e+1=q_1q_2+1=A_2,
\]
so the statement holds. When $k=3$, we have that
\[
S_{\q^{3}}^o+q_3(S_{\q^{1}}^o\cdot S_{\q^{2}}^e) = q_1+q_3+q_3q_1q_2=A_3,
\]
so the statement holds.

Now assume $k\geq4$. Suppose the lemma is satisfied for all $i\in\{2,3,\dotsc,k-1\}$.  Recall that $A_k= q_k(A_{k-1})+A_{k-2}$.

For even $k$, we have that
\begin{align*}
A_k&=q_k(A_{k-1})+A_{k-2}\\
&\geq q_k(S_{\q^{k-1}}^o+q_{k-1}(S_{\q^{k-3}}^o\cdot S_{\q^{k-2}}^e))+S_{\q^{k-3}}^o\cdot S_{\q^{k-2}}^e+1\\
&=q_kS_{\q^{k-1}}^o+q_kq_{k-1}(S_{\q^{k-3}}^o\cdot S_{\q^{k-2}}^e)+S_{\q^{k-3}}^o\cdot S_{\q^{k-2}}^e+1\\
&\geq q_k S_{\q^{k-1}}^o+q_{k-1}S_{\q^{k-2}}^e +S_{\q^{k-3}}^o\cdot S_{\q^{k-2}}^e+1\\
&=q_k S_{\q^{k-1}}^o + S_{\q^{k-1}}^o\cdot S_{\q^{k-2}}^e+1\\
&=S_{\q^{k-1}}^o\cdot S_{\q^{k}}^e+1.
\end{align*}

For odd $k$, we have that
\begin{align*}
A_k&=q_k(A_{k-1})+A_{k-2}\\
&\geq q_k(S_{\q^{k-2}}^o\cdot S_{\q^{k-1}}^e+1) + S_{\q^{k-2}}^o+q_{k-2}(S_{\q^{k-4}}^o\cdot S_{\q^{k-3}}^e)\\
&> q_k(S_{\q^{k-2}}^o\cdot S_{\q^{k-1}}^e) +q_k + S_{\q^{k-2}}^o\\
&= S_{\q^{k}}^o+q_k(S_{\q^{k-2}}^o\cdot S_{\q^{k-1}}^e).\qedhere
\end{align*}
\end{proof}

We use Lemma~\ref{lem:Akbounds} to establish the following proposition.

\begin{proposition}\label{prop:qeven}
Let $\q^k=(q_1,q_2,\dotsc,q_k)\in \mathbb{Z}_{>0}^k$, and let $b\geq 2$. If both $S_{\q^k}^o>b$ and $S_{\q^k}^e>b$, then $A_k>b^2$.
\end{proposition}

\begin{proof}
If $k=0$ or $k=1$, then $S_{\q^k}^e=0$, so the statement is vacuously true. For even $k$ at least $2$, we first note that $S_{\q^k}^o=S_{\q^{k-1}}^o$. Then, using Lemma~\ref{lem:Akbounds}, we have that $A_k\geq S_{\q^{k-1}}^o\cdot S_{\q^k}^e+1>b^2+1$. For odd $k$ at least $3$, we first note that $S_{\q^k}^e=S_{\q^{k-1}}^e$. We have that $q_k+S_{\q^{k-2}}^o=S_{\q^{k}}^o>b$, and it is a simple exercise to show that if the sum of two positive integers is greater than $b$ then the product of these integers is greater than $b-1$. Therefore $q_kS_{\q^{k-2}}^o>b-1$. Then, using Lemma~\ref{lem:Akbounds}, we have that
\[
A_k\geq S_{\q^{k}}^o+q_k(S_{\q^{k-2}}^o\cdot S_{\q^{k-1}}^e)>b+(b-1)b=b^2.\qedhere
\]
\end{proof}

It is well established~\cite{Koshy2002} that for any pair $(\beta,\epsilon)$, where $g=\gcd(\beta,\epsilon)$, $\q=\q_{\beta,\epsilon}$ and $k=\abs{\q}$, we have that
\[
\begin{bmatrix}\beta\\ \epsilon\end{bmatrix} = M^{(k)}_\q \begin{bmatrix}g\\ 0\end{bmatrix}.
\]
In other words, $\beta$ can be written in terms of $g$ and the elements of $\q$ as $\beta=gA_k$.

We are now prepared to prove a result about $F(b^2,\epsilon)$ using Lemma~\ref{lemma:Fintermsofq} and Proposition~\ref{prop:qeven}.

\begin{lemma}\label{lemma:halfareshort}
For each $b\geq2$, there are at least $(b^2+b)/2$ values of $\epsilon\in\{0,1,\dotsc,b^2-1\}$ for which $F(b^2,\epsilon)\leq b+2$.
\end{lemma}
\begin{proof}
We first note that if $\epsilon=kb$ for some $k\in\{0,1,\dotsc,b-1\}$ then $F(b^2,\epsilon)=F(b^2,kb)=F(b,k)\leq b$. This gives $b$ values of $\epsilon$ for which $F(b^2,\epsilon)\leq b$.

For the remaining $b^2-b$ values of $\epsilon$, we will show that either $F(b^2,\epsilon)\leq b+2$ or $F(b^2,b^2-\epsilon)\leq b+2$. Let $\q=\q_{b^2,\epsilon}=(q_1,q_2,\dotsc, q_k)$ as defined in~\eqref{eq:gcdEuclidean}. Assume for sake of contradiction that $F(b^2,\epsilon)>b+2$ and $F(b^2,b^2-\epsilon)>b+2$. Lemma~\ref{lemma:Fintermsofq} then gives that $S_{\q}^e>b$ and $S_{\q}^o>b$, and Proposition~\ref{prop:qeven} thus yields $A_k>b^2$. However, we observed above that $b^2=\gcd(b^2,\epsilon) A_k$, so we thus have $b^2>b^2$, a contradiction. Therefore either $F(b^2,\epsilon)$ or $F(b^2,b^2-\epsilon)$ must be less than $b+2$, meaning that at least half of the values of $\epsilon$ in this case satisfy $F(b^2,\epsilon)\leq b+2$.

Therefore at least $b+(b^2-b)/2=(b^2+b)/2$ values of $\epsilon$ satisfy $F(b^2,\epsilon)\leq b+2$.
\end{proof}

Experimentally the number of values of $\epsilon$ for which $F(b^2,\epsilon) \le b+2$ is greater than the roughly $50\%$ guaranteed by this lemma. In particular, for all $b \le 200$ at least $83\%$ of choices of $\epsilon$ satisfy this condition; the portion is at least $90\%$ for $26 \le b \le 200$ and at least $95\%$ for $72 \le b \le 200$. Nevertheless, the result we are able to prove in Lemma~\ref{lemma:halfareshort} is enough to establish the following cubic lower bound on the number of smooth arithmetical structures on $D_n$.

\begin{proposition}\label{prop:lowerbound}
For all $n\geq 4$, the number of smooth arithmetical structures on $D_n$ is bounded below as follows:
\[
\frac{1}{24}(n^3-3n^2-n-45)\leq \abs{\SArith(D_n)}.
\]
\end{proposition}

\begin{proof}
The result is automatically true for $n\leq 5$ since the lower bound is nonpositive; we prove it under the assumption $n\geq 6$. Consider values of $b$ for which $2 \le b \leq \floor{n/2}-1$. By Lemma~\ref{lemma:halfareshort}, there are at least $(b^2+b)/2$ values of $\epsilon$ with $F(b^2,\epsilon)\leq b+2$. For these values of $\epsilon$, we have that
\[
F(b^2,\epsilon)+b\leq \floor{n/2}-1+2+\floor{n/2}-1\leq n.
\]
Therefore we can set $t=n-F(b^2,\epsilon)-(b-1)$ and have $t\geq 1$. We then set $a=tb^2+\epsilon-b$, noting this ensures that $a \ge b^2-b$ and $a>b$. (The case $a=b=2$ yields $n=3$, which we are not considering here.) Corollary~\ref{Cor:length} then provides two smooth arithmetical structures on $D_n$: one with $\r_x=a$ and $\r_y=b$ and the other with $\r_x=b$ and $\r_y=a$. For each choice of $b$ satisfying $2\leq b\leq \floor{n/2}-1$, we therefore have at least $b^2+b$ smooth arithmetical structures on $D_n$. We can thus compute that
\begin{align*}
\abs{\SArith(D_n)} &\geq \sum_{b=2}^{\floor{n/2}-1} (b^2+ b)\\
&= \frac{\floor{n/2}(\floor{n/2}-1)}{2} + \frac{\floor{n/2}(\floor{n/2}-1)(2\floor{n/2}-1)}{6} -2\\
&= \frac{(\floor{n/2})^3-\floor{n/2}}{3}-2.
\end{align*}

Regardless of whether $\floor{n/2}=(n-1)/2 $ or $\floor{n/2}=n/2 $, this implies the bound in the proposition.
\end{proof}

Propositions~\ref{prop:upperbound} and~\ref{prop:lowerbound} together yield Theorem~\ref{thm:Xnbounds}.

We note that all of the smooth arithmetical structures counted in Proposition~\ref{prop:lowerbound} have $t\geq1$ and $2\leq b\leq\floor{n/2}-1$, whereas there are also smooth arithmetical structures with $t=0$ and/or with $\floor{n/2}\leq b\leq 2n-4$. In fact, experimental data shows that, for each $n$ in the range $4\leq n\leq 200$, the proportion of smooth arithmetical structures on $D_n$ that satisfy $t\geq1$ and $2\leq b\leq\floor{n/2}-1$ is less than $1/4$. Therefore the lower bound in Theorem~\ref{thm:Xnbounds} is not close to being optimal, though it is of the right order, as both the upper and lower bounds in Theorem~\ref{thm:Xnbounds} are cubic in $n$.

\subsection{Bounds on total number of arithmetical structures}
We now use Theorem~\ref{thm:Xnbounds} to obtain upper and lower bounds on the number of arithmetical structures on $D_n$.

\begin{theorem}\label{Thm:overallcount}
For $n\geq 4$, we have that
\[
2 C(n-2)+ C(n-3) \leq \abs{\Arith(D_n)} < 2 C(n-2)+ 702  C(n-3).
\]
\end{theorem}

\begin{proof}
Recall from Theorem~\ref{thm:smoothtooverall} that the number of arithmetical structures on $D_n$ is given by
\[
\abs{\Arith(D_n)} = 2C(n-2)+\sum_{m=4}^nB(n-3,n-m)\abs{\SArith(D_m)},
\]
where $C(n)$ is the $n$-th Catalan number and $B(n,k) = \frac{n-k+1}{n+1}\binom{n+k}{n}$.

We first establish an upper bound on $\abs{\Arith(D_n)}$. Since $\log(n-3)<\frac{n-3}{2}$, Proposition~\ref{prop:upperbound} implies that
\[
\abs{\SArith(D_m)}< \frac{5}{3}m^3-7 m^2+\frac{37}{3}m-9.
\]
We therefore have that
\begin{align*}
\abs{\Arith(D_n)}-2C(n-2) &<\sum_{m=4}^{n}B(n-3,n-m)\left(\frac{5}{3}m^3-7 m^2+\frac{37}{3}m-9\right)\\
& =\frac{12 (n-3) \left(117 n^3-509 n^2+804 n-460\right) }{(n-2) (n-1) n (n+1) (n+2)}\binom{2 n-7}{n-3}\\
& =\frac{6(n-3)(117 n^3-509 n^2+804 n-460)}{(n-1)n(n+1)(n+2)}C(n-3)\\
& \leq 702 C(n-3).
\end{align*}
Here the equality on the second line follows from standard combinatorial identities as verified by a computer algebra system.  The last inequality follows by showing that the coefficient of $C(n-3)$ is an increasing function of $n$ for $n\geq3$ whose limit is $702$.

We now establish a lower bound on $\abs{\Arith(D_n)}$. From Proposition~\ref{prop:lowerbound}, we have that
\[
\abs{\SArith(D_n)}\geq \frac{1}{24}(n^3-3n^2-n-45).
\]
Therefore it follows that
\begin{align*}
\abs{\Arith(D_n)}-2C(n-2) &\geq \frac{1}{24}\sum_{m=4}^{n}B(n-3,n-m)(m^3-3m^2-m-45)\\
& =\frac{1}{24}\sum_{m=4}^{n}\frac{m-2}{n-2}\binom{2n-m-3}{n-3}(m^3-3m^2-m-45)\\
& = \frac{3 (n - 3) (8 n^3 - 68 n^2 + 69 n - 30)}{2(n-1)n(n+1)(n+2)}C(n-3)\\
&\geq C(n-3) \text{ when $n\geq 9$.}
\end{align*}
Here the equality on the third line follows from standard combinatorial identities as verified by a computer algebra system. The last inequality follows by showing that the coefficient of $C(n-3)$ is an increasing function of $n$ for $n\geq9$ and is greater than $1$ when $n=9$. One can also check directly that $\abs{\Arith(D_n)}-2C(n-2)\geq C(n-3)$ when $4\leq n\leq 8$. The theorem thus follows.
\end{proof}

Since $C(n-3)\leq C(n-2)$, Theorem~\ref{Thm:overallcount} implies that $2 C(n-2) \leq \abs{\Arith(D_n)}\leq  704 C(n-2)$. Thus $\abs{\Arith(D_n)}$ has the same growth rate as the Catalan numbers.

\section{Critical groups}\label{sec:critical}

We next investigate critical groups of arithmetical structures on bidents. We first show that all such critical groups are cyclic. Consequently, the problem of understanding critical groups of arithmetical structures on $D_n$ reduces to that of understanding the orders of these groups. We then completely characterize the groups that occur as critical groups of arithmetical structures on $D_n$.

Before discussing our results, we present some experimental data.
Table~\ref{table:dist} gives the number of arithmetical structures on $D_n$ whose critical group has order $m$ for $n$ in the range $4\leq n\leq 12$. Notice that most rows of the table have gaps, i.e.\ for fixed $n$ there are positive integers $m<M$ for which there is no arithmetical structure on $D_n$ with critical group of order $m$ but there is an arithmetical structure on $D_n$ with critical group of order $M$. There are no such gaps in the columns of the table since, as we will see, any descendant of an arithmetical structure under subdivision has isomorphic critical group.

\begin{table}
\centering
\resizebox{\textwidth}{!}{
\begin{tabular}{|c||c|c|c|c|c|c|c|c|c|c|c|c|c|c|c|c|c|c|c|}
\hline
\diagbox{$n$}{$m$} & 1     & 2    & 3    & 4    & 5    & 6   & 7   & 8  & 9  & 10 & 11 & 12 & 13 & 14 & 15 &16&17&18&19 \\
\hline
\phantom{1}4   & \phantom{232,0}10     & \phantom{65,53}3    & \phantom{39,92}1    & \phantom{15,31}0    & \phantom{15,43}0    & \phantom{1,83}0   & \phantom{4,85}0   & \phantom{1,12}0  & \phantom{1,14}0  & \phantom{9}0  & \phantom{53}0  & 0  & \phantom{15}0  & 0  & \phantom{4}0 & 0  & 0  & 0  & 0 \\ \hline
\phantom{1}5   & \phantom{232,0}32    & \phantom{65,53}8   & \phantom{39,92}5    & \phantom{15,31}0    & \phantom{15,43}1    & \phantom{1,83}0   & \phantom{1,83}0   & \phantom{1,12}0  & \phantom{1,14}0  & \phantom{9}0  & \phantom{53}0  & 0  & \phantom{15}0  & 0  & \phantom{4}0 & 0  & 0  & 0  & 0 \\ \hline
\phantom{1}6   & \phantom{232,}116   & \phantom{65,5}31   & \phantom{39,9}18   & \phantom{15,31}5    & \phantom{15,43}5    & \phantom{1,83}0   & \phantom{1,83}1   & \phantom{1,12}0  & \phantom{1,14}0  & \phantom{9}0  & \phantom{53}0  & 0  & \phantom{15}0  & 0  & \phantom{4}0  & 0  & 0  & 0  & 0\\ \hline
\phantom{1}7   & \phantom{232,}400   & \phantom{65,}108  & \phantom{39,9}65   & \phantom{15,3}22   & \phantom{15,4}20   & \phantom{1,83}0   & \phantom{1,83}4   & \phantom{1,12}0  & \phantom{1,14}1  & \phantom{9}0  & \phantom{53}0  & 0  & \phantom{15}0  & 0  & \phantom{4}0 & 0  & 0  & 0  & 0 \\ \hline
\phantom{1}8   & \phantom{23}1,406  & \phantom{65,}384  & \phantom{39,}236  & \phantom{15,3}84   & \phantom{15,4}79   & \phantom{1,83}3   & \phantom{1,8}18  & \phantom{1,12}2  & \phantom{1,14}5  & \phantom{9}0  & \phantom{53}1  & 0  & \phantom{15}0  & 0  & \phantom{4}0 & 0  & 0  & 0  & 0 \\ \hline
\phantom{1}9   & \phantom{23}4,980  & \phantom{6}1,366 & \phantom{39,}848  & \phantom{15,}308  & \phantom{15,}300  & \phantom{1,8}20  & \phantom{1,8}77  & \phantom{1,1}12 & \phantom{1,1}20 & \phantom{9}0  & \phantom{53}6  & 0  & \phantom{15}1  & 0  & \phantom{4}0 & 0  & 0  & 0  & 0 \\ \hline
10  & \phantom{2}17,794 & \phantom{6}4,885 & \phantom{3}3,050 & \phantom{1}1,131 & \phantom{1}1,122 & \phantom{1,}101 & \phantom{1,}314 & \phantom{1,1}59 & \phantom{1,1}77 & \phantom{9}2  & \phantom{5}29 & 0  & \phantom{15}7  & 0  & \phantom{4}1 & 0  & 0  & 0  & 0 \\ \hline
11&\phantom{2}64,042& 17,566& 11,009& \phantom{1}4,158& \phantom{1}4,166& \phantom{1,}450& 1,245& \phantom{1,}264& \phantom{1,}296& 16& 128& 0& \phantom{1}35& 0& \phantom{4}8& 0& 1& 0& 0\\ \hline
12&232,018& 63,530& 39,920& 15,314& 15,431& 1,883& 4,856& 1,120& 1,142& 93& 537& 0& 156& 2& 44& 0& 9& 0& 1\\ \hline
\end{tabular}
}
\caption{The distribution of critical group orders $m=\abs{\mathcal{K}(D_n; \vd, \vr)}$ for $4\leq n\leq 12$.}\label{table:dist}
\end{table}

In this section, we consider the following two dual questions that explain the distribution of possible critical group orders, including the gaps observed above.
\begin{enumerate}[(1)]
\item Given some $n$, what is the maximal order of a critical group of an arithmetical structure on $D_n$?
\item Given some $m$, what is the minimal number of vertices $n$ such that there is an arithmetical structure on $D_n$ with critical group of order $m$?
\end{enumerate}
Note that the first question is asking for the last nonzero entry in each row of Table~\ref{table:dist} and the second question is asking for the first nonzero entry in each column of Table~\ref{table:dist}.

\subsection{Basic properties}

Recall that the \emph{critical group} $\mathcal{K}(D_n; \vd, \vr)$ of an arithmetical structure $(\vd,\vr)$ is the torsion part of the cokernel of the matrix $L(D_n,\vd)\coloneqq\diag(\vd)-A$.

\begin{proposition}\label{prop:cyclic}
Let $n\geq 3$. The critical group of any arithmetical structure on $D_n$ is cyclic.
\end{proposition}

The proof of Proposition~\ref{prop:cyclic} relies on the fact that the torsion part of the cokernel of $L(D_n,\vd)$ is given by the direct sum $\bigoplus_{i=1}^{n-1} \mathbb{Z}/{\alpha_i}\mathbb{Z}$, where $\alpha_i$ is the $i$-th diagonal entry of the Smith normal form of $L(D_n,\vd)$. It is well known that the product $\alpha_1\alpha_2\dotsm\alpha_i$ is given by the greatest common divisor of the  $i\times i$ minors of $L(D_n,\vd)$. In particular, since $L(D_n,\vd)$ has rank $n-1$, we have that $\alpha_1\leq \alpha_2\leq \dotsb\leq\alpha_{n-1}$.  For more details about the Smith normal form of a matrix, see~\cite{SNF}.

\begin{proof}[Proof of Proposition~\ref{prop:cyclic}]
Fix $n$, and let $(\vd,\vr)$ be an arithmetical structure on $D_n$. We first find an $(n-2)\times(n-2)$ minor of $L(D_n,\vd)$ with value $\pm1$. Such a minor can be obtained by deleting the columns associated to vertices $v_x$ and $v_y$ and deleting the rows associated to vertices $v_x$ and $v_\ell$. The greatest common divisor of the $(n-2)\times(n-2)$ minors is thus $1$, which implies that $\alpha_1\alpha_2\dotsm\alpha_{n-2}=1$.  Therefore the Smith normal form of $L(D_n,\vd)$ has at most one nontrivial diagonal entry, so the critical group $\mathcal{K}(D_n; \vd, \vr)$ is cyclic.
\end{proof}

It remains to consider the order of the critical group of a given arithmetical structure.  Since $D_n$ is a tree, the following proposition of Lorenzini given in~\cite{L89} applies.

\begin{proposition}[{\cite[Corollary 2.5]{L89}}] \label{prop: trees}
For any tree $T$ and any arithmetical structure $(\vd, \vr)$ on $T$, the order of the corresponding critical group is given by
\[
\abs{\mathcal{K}(T; \vd, \vr)}=\prod_{v\in T} r_v^{\deg(v)-2}.
\]
\end{proposition}

Applying this proposition to bidents, we obtain the following corollaries.

\begin{corollary}\label{cor: size of K}
Let $n\geq4$, let $(\vd, \vr)$ be an arithmetical structure on $D_n$, and let $\ell=n-3$. The corresponding critical group has order given by
\[
\abs{\mathcal{K}(D_n; \vd, \vr)}=\frac{r_0}{r_x r_y r_\ell}.
\]
\end{corollary}

\begin{corollary}\label{cor:d3}
Let $(\vd,\vr)$ be an arithmetical structure on $D_3$. Then $\mathcal{K}(D_3;\vd,\vr)$ is the trivial group.
\end{corollary}

\begin{proof}
Proposition~\ref{prop: trees} gives that
\[
\abs{\mathcal{K}(D_3;\vd,\vr)}=\frac{1}{r_xr_y}.
\]
Since $r_x$ and $r_y$ are integers, we must have $\abs{\mathcal{K}(D_3;\vd,\vr)}=1$; hence $\mathcal{K}(D_3;\vd,\vr)$ is the trivial group.
\end{proof}

\begin{corollary}\label{cor:paths}
Let $n\geq4$, and let $(\vd,\vr)$ be an arithmetical structure on $D_n$. If $d_x=1$ or $d_y=1$, then $\mathcal{K}(D_n;\vd,\vr)$ is the trivial group.
\end{corollary}

\begin{proof}
If $d_x=1$, then $r_x=r_0$. Corollary~\ref{cor: size of K} then gives that
\[
\abs{\mathcal{K}(D_n;\vd,\vr)}=\frac{r_0}{r_x r_y r_\ell}=\frac{1}{r_1r_\ell}.
\]
Since $r_1$ and $r_\ell$ are integers, we must have $\abs{\mathcal{K}(D_n;\vd,\vr)}=1$; hence $\mathcal{K}(D_n;\vd,\vr)$ is the trivial group. The same argument applies if $d_y=1$.
\end{proof}

We next show that the smoothing and subdivision operations on bidents described in Section~\ref{sec:smooth} preserve the critical group.

\begin{lemma}\label{lem:smoothingpreservesorder}
Let $(\vd,\vr)$ be an arithmetical structure on $D_n$ where $n\geq4$, and let $(\vd',\vr')$ be an ancestor of $(\vd,\vr)$ on some $D_m$, where $3\leq m\leq n$. Then $\mathcal{K}(D_m;\vd',\vr')\cong\mathcal{K}(D_n;\vd,\vr)$.
\end{lemma}

\begin{proof}
Since, by Proposition~\ref{prop:cyclic}, all critical groups of arithmetical structures on $D_n$ are cyclic, it suffices to show that all possible smoothing operations on bidents preserve the order of the critical group. When smoothing at a vertex of degree $2$ in the tail, the values of $r_i$ for $i \in \{x,y,0\}$ are unchanged, and $r'_{\ell-1}=r_{\ell}$, so Corollary~\ref{cor: size of K} implies that the order of the critical group is unchanged. Smoothing at the endpoint vertex $v_{\ell}$ with $\ell\geq 2$, we have that $r'_{\ell-1}=r_{\ell-1}=r_{\ell}$, and hence smoothing at vertex $v_{\ell}$ does not change the expression in Corollary~\ref{cor: size of K}. If $\ell=1$, smoothing at $v_1$ in $D_4$ results in the graph $D_3$, and, since we must have had that $r_1=r_0$, it follows that $\abs{\mathcal{K}(D_4;\vd,\vr)}=\frac{r_0}{r_xr_yr_1}=\frac{1}{r_xr_y}=\abs{\mathcal{K}(D_3;\vd',\vr')}$.
\end{proof}

Because of the preceding lemma, most of the extremal examples in this section involve critical groups of \emph{smooth} arithmetical structures. In the next lemma, we make an observation about the possible orders of critical groups of such structures on $D_n$.

\begin{lemma}\label{lem:crit group inequality}
Suppose $n\geq 4$, and let $(\vd,\vr)$ be a smooth arithmetical structure on $D_n$. Let $\ell=n-3$. Then
\[
\abs{\mathcal{K}(D_n; \vd, \vr)} \leq (n-3)\left(\frac{1}{\rr_y\rr_\ell}+\frac{1}{\rr_x\rr_\ell}\right)+\frac{1}{\rr_x\rr_y}.
\]
\end{lemma}
\begin{proof}
By Lemma~\ref{lem:smooth}, a smooth arithmetical structure on $D_n$ must satisfy $r_i-r_{i+1}\leq r_0-r_1$
for all $1\leq i \leq \ell-1$. We thus have that
\begin{align*}
\rr_0 & = (\rr_0-\rr_{1}) + (\rr_{1}-\rr_{2}) + \dotsb + (\rr_{\ell-1}-\rr_\ell) + \rr_\ell \\
& \leq \ell(\rr_0-\rr_1) + \rr_\ell.
\end{align*}
Furthermore, Proposition~\ref{prop:centraldis1} says that a smooth arithmetical structure must satisfy $\rr_0=\rr_x+\rr_y+\rr_1$, so therefore $\rr_0-\rr_1=\rr_x+\rr_y$. Using Corollary~\ref{cor: size of K}, we then have that
\[
\abs{\mathcal{K}(D_n; \vd, \vr)}=\frac{ \rr_0}{ \rr_{x}\rr_{y}\rr_{\ell}}  \leq \frac{\ell(\rr_x+\rr_y)+\rr_\ell}{\rr_{x}\rr_{y}\rr_{\ell}}= (n-3)\left(\frac{1}{\rr_y\rr_\ell}+\frac{1}{\rr_x\rr_\ell}\right)+\frac{1}{\rr_x\rr_y}.\qedhere
\]
\end{proof}

\subsection{Maximal order of a critical group of an arithmetical structure on \texorpdfstring{$\bm{D_n}$}{D\_n}}

The main result of this subsection gives the maximal order of a critical group of an arithmetical structure on $D_n$ and shows that it is realized by a unique arithmetical structure. This maximal order is always odd, so it is natural to also ask for the maximal even order of a critical group of an arithmetical structure on $D_n$. The following theorem addresses both of these questions.

\begin{theorem}\label{thm:criticalgp}
Let $n\geq 4$. The following properties hold:
\begin{enumerate}[$(a)$]
\item \label{cond:cardinality} The maximal order of a critical group of an arithmetical structure on $D_n$ is $2n-5$. Moreover, there is a unique arithmetical structure on $D_n$ whose critical group has order $2n-5$.
\item \label{cond:even} The maximal even order of a critical group of an arithmetical structure on $D_n$ is given by
\[
\max\{\abs{\mathcal{K}(D_n;\vd,\vr)} : \abs{\mathcal{K}(D_n;\vd,\vr)}\text{ is even}\} =\begin{cases}
6k-4 & \text{if } n=4k \text{ or } n=4k+1 \\
6k-2 & \text{if } n=4k+2 \text{ or } n=4k+3.
\end{cases}
\]
\end{enumerate}
\end{theorem}

\begin{proof}
Fix $n$, and let $(\vd,\vr)$ be an arithmetical structure on $D_n$. In considering part $(a)$, we first show that $2n-5$ is the maximal order of a critical group of a \emph{smooth} arithmetical structure on $D_n$. Using Lemma~\ref{lem:crit group inequality}, for a smooth arithmetical structure,
\begin{equation}
\abs{\mathcal{K}(D_n;\vd,\vr)}=\frac{r_0}{r_xr_yr_{\ell}}\leq (n-3)\left(\frac{1}{\rr_y\rr_\ell}+\frac{1}{\rr_x\rr_\ell}\right)+\frac{1}{\rr_x\rr_y} \leq 2n-5.   \label{eq:criticalgroupbound}
\end{equation}
To show that equality will hold for a unique smooth arithmetical structure on $D_n$, we note that in order for the second inequality in~\eqref{eq:criticalgroupbound} to be an equality, we must have $r_x=r_y=r_{\ell}=1$. If the first inequality in~\eqref{eq:criticalgroupbound} is also an equality, this then implies that $r_0=2n-5$. By Proposition~\ref{prop:abcuniqueness}, these choices for $r_x, r_y$, and $r_0$ do in fact determine a unique smooth arithmetical structure on some bident, namely that with $r_x=r_y=1$ and $r_i=2(n-i)-5$ for $0\leq i \leq \ell$. Noting that this gives $r_{n-3}=1$, this smooth arithmetical structure is on $D_n$.

Now suppose $(\vd,\vr)$ is a \emph{non-smooth} arithmetical structure $(\vd',\vr')$ on $D_n$. If $d_x=1$ or $d_y=1$, then Corollary~\ref{cor:paths} gives that $\abs{\mathcal{K}(D_n;\vd,\vr)}=1<2n-5$. If $d_x\geq2$ and $d_y\geq2$, then Lemma~\ref{lem:uniquesmoothancestor} gives that, by performing a sequence of smoothing operations, we obtain a smooth arithmetical structure on $D_N$ with $3\leq N< n$, and by Lemma~\ref{lem:smoothingpreservesorder} the critical group of this smooth arithmetical structure is isomorphic to that of $(\vd,\vr)$. In this case $N\leq n-1$, and hence, using~\eqref{eq:criticalgroupbound}, we have that
\[
\abs{\mathcal{K}(D_n;\vd,\vr)}=\abs{\mathcal{K}(D_N;\vd',\vr')}\leq 2N-5\leq 2(n-1)-5<2n-5.
\]
Therefore the smooth arithmetical structure described above is the unique arithmetical structure on $D_n$ with critical group of order $2n-5$, and there is no arithmetical structure on $D_n$ with larger order critical group.

Now consider part $(b)$. We first see that there exist arithmetical structures on bidents with critical groups of the orders given in the theorem. Let $k$ be a positive integer. Consider the smooth arithmetical structure determined by $r_x=2$, $r_y=1$, and $r_0=12k-8$. By Proposition~\ref{prop:abcuniqueness}, we have that $r_i=12k-8-3i$ for all $0\leq i \leq 4k-3$. Noting that $r_{4k-3}=1$, this gives an arithmetical structure on $D_{4k}$ with critical group of order $6k-4$.  Subdividing at $v_{4k-3}$, we get another (non-smooth) arithmetical structure on $D_{4k+1}$ with critical group of the same order. Next consider the smooth arithmetical structure determined by $r_x=2$, $r_y=1$, and $r_0=12k-4$. By Proposition~\ref{prop:abcuniqueness}, we have that $r_i=12k-4-3i$ for all $0\leq i \leq 4k-2$ (and thus $r_{4k-2}=2$) and $r_{4k-1}=1$. It is easily checked that this is a smooth arithmetical structure on $D_{4k+2}$ with critical group of order $6k-2$. Subdividing at $v_{4k-1}$, we get another (non-smooth) arithmetical structure on $D_{4k+3}$ with critical group of the same order.

Now let us see that these are the maximal even orders of critical groups of arithmetical structures on bidents. First consider smooth arithmetical structures. Notice that if $r_x=r_y=r_\ell=1$, then, using Proposition~\ref{prop:centraldis1} and Lemma~\ref{lem:gcd}$(b)$, we have that $\gcd(r_0,r_0-2)=\gcd(r_0,r_1)=r_{\ell}=1$, so therefore $r_0$ is odd. By Corollary~\ref{cor: size of K}, this implies that $\abs{\mathcal{K}(D_n; \vd, \vr)}$ is odd. Therefore, to obtain a critical group of even order, at least one of $r_x,r_y$ and $r_\ell$ must be greater than 1. Consider Lemma~\ref{lem:crit group inequality} in this new context. The extremal case occurs when $r_x=2$ and $r_y=r_\ell=1$ (or, symmetrically, when $r_y=2$ and $r_x=r_\ell=1$) and  thus gives that
\[
\abs{\mathcal{K}(D_n;\vd,\vr)}\leq \begin{cases}
6k-4 & \text{if } n=4k \\
6k-\frac{5}{2} & \text{if } n=4k+1 \\
6k-1 & \text{if } n=4k+2 \\
6k+\frac{1}{2} & \text{if } n=4k+3.
\end{cases}
\]

Since we are only considering even order critical groups, this inequality implies that
\[
\max\{\abs{\mathcal{K}(D_n;\vd,\vr)} : \abs{\mathcal{K}(D_n;\vd,\vr)}\text{ is even}\}\leq \begin{cases}
6k-4 & \text{if } n=4k \\
6k-4 & \text{if } n=4k+1 \\
6k-2 & \text{if } n=4k+2 \\
6k & \text{if } n=4k+3.
\end{cases}
\]
This proves that, for $n=4k$, $n=4k+1$, and $n=4k+2$, there are no smooth arithmetical structures whose critical group has a larger even order than those arithmetical structures found above. Moreover, if there were some non-smooth arithmetical structure whose critical group was a larger even order, then it would have a smooth ancestor on some $D_N$ with $3\leq N<n$ with the same order critical group. Inductively, we see that there are no such smooth arithmetical structures on these graphs.

It remains to show that there is no critical group  of order larger than $6k-2$ when $n=4k+3$. First note that such a structure would have to be smooth since if it were not smooth it would either have $d_x=1$ or $d_y=1$, in which case its critical group would be trivial by Corollary~\ref{cor:paths}, or have a smooth ancestor on some $D_N$ with $N<4k+3$ with the same order critical group, but the previous paragraph shows that this is impossible. Since the previous paragraph also shows that we cannot have a smooth arithmetical structure on $D_n$ with $n=4k+3$ whose critical group has even order larger than $6k$, it only remains to rule out the possibility of a smooth arithmetical structure on $D_n$ with $n=4k+3$ whose critical group has order $6k$.

Consider the case when $n=4k+3$, and suppose there exists a smooth arithmetical structure on $D_n$ with critical group of order $6k$. Then, by Corollary~\ref{cor: size of K}, we have that $\frac{r_0}{r_xr_yr_\ell}=6k$. It is important to note that, by Lemma~\ref{lem:crit group inequality}, our choices for $r_x$, $r_y$, and $r_\ell$ are limited. Notice that if $r_x,r_y\geq 2$ and $r_\ell\geq 1$, then
\[
\abs{\mathcal{K}(D_n;\vd,\vr)}\leq((4k+3)-3)\left(\frac{1}{2\cdot1} + \frac{1}{2\cdot 1}\right)+\frac{1}{2\cdot2}=4k+\frac{1}{4},
\]
which is less than $6k$ for any positive integer $k$. Similar results hold for: $r_x, r_\ell \geq 2$ and $r_y\geq1$; $r_x, r_\ell \geq 1$ and $r_y\geq3$;  and $r_x,r_y\geq 1$ and $r_\ell\geq 2$. Thus the only case we consider is when $r_y=r_\ell=1$ and $r_x=2$ (or, symmetrically, when $r_x=r_{\ell}=1$ and $r_y=2$). However, if $r_x=2$, $r_y=1$, and $r_0=12k$, Proposition~\ref{prop:centraldis1} and Lemma~\ref{lem:gcd}$(b)$ would then give $r_\ell=\gcd(12k,12k-3)=3$. Therefore it is not possible to have $r_y=r_\ell=1$ and $r_x=2$ together with a critical group of order $6k$.
\end{proof}

Notice that in the previous theorem the maximal even orders of critical groups are not divisible by $6$. We could continue our line of questioning by asking, for a given $n$, for the maximal order of a critical group of an arithmetical structure on $D_n$ that is $0$ modulo $6$.  However, as we will see, this is more naturally addressed by first fixing the order $m$ of a critical group and finding the smallest $n$ for which there is an arithmetical structure on $D_n$ whose critical group has order $m$. The next subsection addresses this question.

\subsection{Minimal number of vertices for a given order critical group}

We now consider the dual question to that of the previous subsection. Specifically, we fix $m$ and ask for which values of $n$ there is an arithmetical structure on $D_n$ whose critical group has order $m$.

We note that since, by Lemma~\ref{lem:smoothingpreservesorder}, subdivision of arithmetical structures on bidents preserves critical groups, it will be the case that if there is an arithmetical structure on $D_n$ with a given critical group then there will also be an arithmetical structure on $D_N$ with the same critical group for all $N \ge n$.  Therefore it suffices to ask for the minimal $n$ for which there is an arithmetical structure on $D_n$ with critical group of order $m$. Lemma~\ref{lem:smoothingpreservesorder} also shows that smoothing along the tail of a bident preserves the critical group, so therefore an arithmetical structure with critical group of order~$m$ on $D_n$ with minimal value of $n$ must satisfy the equivalent conditions of Lemma~\ref{lem:smooth}. Moreover, since critical groups of arithmetical structures with $d_x=1$ or $d_y=1$ are trivial by Corollary~\ref{cor:paths}, it suffices to restrict attention to smooth arithmetical structures on bidents.

We begin with the following result, which guarantees the existence of smooth arithmetical structures on bidents with certain order critical groups. In particular, for all $m$, we can choose $k=m+1$ in the following theorem and get a smooth arithmetical structure on $D_{m+2}$ with critical group of order $m$.

\begin{theorem}\label{T:construction}
Let $m,k\geq2$ be a pair of relatively prime integers, and let $r$ be the least residue of $m$ modulo $k$.  If $n= m-\floor*{\frac{m}{k}} + F(k,k-r)$, then there exists a smooth arithmetical structure on $D_n$ with critical group of order $m$.
\end{theorem}

\begin{proof}
Write $m=qk+r$ where $0\leq r<k$, and let $r_x=1$, $r_y=k-1$, and $r_0=m(k-1)=(qk+r-q)k-r$. By Proposition~\ref{prop:abcuniqueness}, this determines a unique smooth arithmetical structure on some $D_n$. By Proposition~\ref{prop:centraldis1}, we must have $r_1=r_0-(r_x+r_y)$, so $r_1=(qk+r-q-1)k-r$.  Moreover, $r_{i}=(qk+r-q-i)k-r$ for all $i$ in the range $0\leq i\leq qk+r-q-1$.  In particular, $r_{qk+r-q-2}=2k-r$ and $r_{qk+r-q-1}=k-r$.  We thus have that
\[
n=F(r_0,r_1)+2=qk+r-q-2+F(2k-r,k-r)+2=m-\floor*{\frac{m}{k}} + F(k,k-r).
\]
This is the desired value of $n$.

Using Lemma~\ref{lem:gcd}$(b)$ and the fact that $m$ and $k$ are relatively prime, we further have that
\[
r_{\ell}=\gcd(r_0,r_1)=\gcd(m(k-1),m(k-1)-k)=\gcd(m(k-1),k)=\gcd(m,k)=1.
\]
Therefore the order of the critical group of this arithmetical structure is $\frac{r_0}{r_xr_yr_\ell} = \frac{m(k-1)}{(k-1)}=m$, as desired.
\end{proof}

The above construction is optimal in the sense that, for fixed $m$, it can always be used to produce an arithmetical structure with a critical group of order $m$ on the smallest possible bident, as the following proposition shows. For ease of exposition, we define $N(m,k)=m-\floor*{\frac{m}{k}} + F(k,k-r)$.

\begin{proposition}\label{prop:optimal}
Let $m\geq2$. If $n$ is the smallest integer for which there is an arithmetical structure on $D_n$ with critical group of order $m$, then there is some $k\geq2$ coprime to $m$ for which $n=N(m,k)$.
\end{proposition}

\begin{proof}
First note that it follows from Lemma~\ref{lem:crit group inequality} that any smooth arithmetical structure on $D_n$ with $r_x \geq 2$ and $r_y \geq 2$ will have critical group of order at most $n-2$.  This implies that any smooth arithmetical structure of this type with critical group of order $m$ will have at least $m+2$ vertices. This is no better than using Theorem~\ref{T:construction} with $k=m+1$. Because of the symmetry of the bident, we can therefore assume $r_x=1$.

Second, note that if $r_{\ell}\geq2$ it then follows from Lemma~\ref{lem:crit group inequality} that any smooth arithmetical structure on $D_n$ will have critical group of order at most $n-2$. This implies that any smooth arithmetical structure of this type with critical group of order $m$ will have at least $m+2$ vertices, which is again no better than using Theorem~\ref{T:construction} with $k=m+1$. Therefore it suffices to restrict attention to smooth arithmetical structures with $r_{\ell}=1$.

Thus, in looking for the smallest bident on which there is an arithmetical structure whose critical group has order $m$, it suffices to consider smooth arithmetical structures with $r_x=1$ and $r_{\ell}=1$. All such structures are of the form constructed in Theorem~\ref{T:construction} for some $m$ and $k$, and since $\gcd(m,k)=r_{\ell}=1$ we must have that $k$ is coprime to $m$.
\end{proof}

We claim moreover that choosing the \emph{smallest} $k\geq2$ coprime to $m$ and using the construction of Theorem~\ref{T:construction} usually gives the smallest $n$ for which there is an arithmetical structure on $D_n$ with critical group of order $m$. Before stating this result, we prove a lemma we will use repeatedly.

\begin{lemma}\label{L:Lbound}
For any relatively prime pair of integers $m,k\geq2$, we have the following bounds on $N(m,k)$:
\begin{enumerate}[$(a)$]
\item $N(m,k) > \frac{(k-1)m}{k} + 2$.
\item If $m \equiv 1 \pmod{k}$, then $N(m,k) = \frac{(k-1)m+1}{k}+k$.  Otherwise, $N(m,k) \le \frac{(k-1)m+k-1}{k}+\frac{k+1}{2}$.
\end{enumerate}
\end{lemma}

\begin{proof}
For $(a)$, since $m$ and $k$ are relatively prime, $\floor*{\frac{m}{k}}<\frac{m}{k}$. Also, since $r<k$, we have that $k-r>0$, and hence $F(r,k-r)\geq2$. It follows that $N(m,k) > \frac{(k-1)m}{k} + 2$.

For $(b)$, if $m\equiv1\pmod{k}$, then $\floor*{\frac{m}{k}}=\frac{m-1}{k}$. Also, if $m\equiv1\pmod{k}$, then $r=1$, and hence $F(k,k-r)=F(k,k-1)=k$. Therefore $N(m,k) = \frac{(k-1)m+1}{k}+k$. Otherwise, we have that $\floor*{\frac{m}{k}}\geq\frac{m-(k-1)}{k}$ and, by Lemma~\ref{lem:Fbasics}$(d)$, that $F(k,k-r) \leq\frac{k+1}{2}$. Hence $N(m,k)\leq \frac{(k-1)m+k-1}{k}+\frac{k+1}{2}$.
\end{proof}

We now state the main theorem of this subsection, which identifies the smallest value of $n$ for which there is a critical group of order $m$.
\begin{theorem}\label{thm:critgporder}
Let $m\geq2$ be an integer.
With the exceptions of $m\in\{6,210\}$, the smallest $n$ for which there is an arithmetical structure on $D_n$ with critical group of order $m$ is given by $N(m,k)$, where $k$ is the smallest integer greater than $1$ that is coprime to $n$ and the structure is obtained by the construction in the proof of Theorem~\ref{T:construction}. For $m=6$, the smallest such $n$ is $8$, and for $m=210$, the smallest such $n$ is $200$.
\end{theorem}

\begin{proof}
By Proposition~\ref{prop:optimal}, it suffices to find the $k\geq2$ coprime to $m$ that minimizes $N(m,k)$. The proof breaks into cases according to the smallest prime that does not divide $m$.

We begin by considering the case where $m$ is odd.  In this case we compute that $N(m,2) = m-\floor*{\frac{m}{2}} + F(2,1) = \frac{m+5}{2}$.  The value of $\frac{(k-1)m}{k}$ is increasing in $k$, so if $k \ge 3$ it follows from Lemma~\ref{L:Lbound} that $N(m,k) > \frac{(k-1)m}{k} + 2 \ge \frac{2m}{3}+2$. Since $\frac{m+5}{2}\leq\frac{2m}{3}+2$ for all $m\geq3$, we therefore have $N(m,2) < N(m,k)$. Thus the construction of Theorem~\ref{T:construction} is optimal with $k=2$, and the smallest bident that has an arithmetical structure with critical group of order $m$ has $n=\frac{m+5}{2}$. Note that this arithmetical structure with $n=\frac{m+5}{2}$ is the same as the one with $m=2n-5$ in Theorem~\ref{thm:criticalgp}$(a)$.

We next consider values of $m$ that are even but not divisible by $3$.  If $m \equiv 1\pmod{3}$, Lemma~\ref{L:Lbound} gives that $N(m,3)=\frac{2m+10}{3}$. If $m \equiv 2\pmod{3}$, we compute that
\[
N(m,3)=m-\floor*{\frac{m}{3}}+F(3,1)=m-\frac{m-2}{3}+2=\frac{2m+8}{3}.
\]
Because $m$ is even and $\gcd(m,k)=1$, we know that if $k \ne 3$ then $k \ge 5$. It follows from Lemma~\ref{L:Lbound} that for such $k$ we have $N(m,k) > \frac{(k-1)m}{k}+2 \geq \frac{4m}{5}+2$.
As long as $m \ge 10$, we then have that $N(m,3)\leq\frac{2m+10}{3}\leq\frac{4m}{5}+2<N(m,k)$. If $m \equiv 2\pmod{3}$, we can improve this bound by saying that $N(m,3)\leq\frac{2m+8}{3}\leq\frac{4m}{5}+2<N(m,k)$ as long as $m\ge5$.  This shows that the construction of Theorem~\ref{T:construction} is optimal with $k=3$ except possibly in the case $m=4$. In that case, we check directly that $N(4,3)=5$ and $N(4,k) > 5$ for all $k \ge 5$, so the result also holds then.  Note that these arithmetical structures are the same as those constructed in Theorem~\ref{thm:criticalgp}$(b)$.

Suppose $m$ is a multiple of $6$ but not a multiple of $5$.  Lemma~\ref{L:Lbound} gives that $N(m,5) \le \frac{4m+19}{5}$ if $m\not\equiv1\pmod{5}$ and $N(m,5) = \frac{4m+26}{5}$ if $m \equiv 1\pmod{5}$.  Let $k\neq5$ be relatively prime to $m$; we must then have $k \ge 7$. Lemma~\ref{L:Lbound} thus gives that $N(m,k) > \frac{(k-1)m}{k} + 2\geq\frac{6m}{7}+2$.  It follows that $N(m,5) < N(m,k)$ for $m\geq 32$ when $m\not\equiv1\pmod{5}$ and $m\geq 56$ when $m\equiv1\pmod{5}$. This leaves the only possible exceptions in this case as $m=6,12,18,24,36$. For $m=6$, we check directly that $F(6,5)=10$, $F(6,7)=8$, and $F(6,k)>8$ for all $k\geq11$. Therefore this is an exception. For $m=12$, we check directly that $F(12,5)=13$ and $F(12,k)>13$ for all $k\geq7$, so this is not an exception. For $m=18,24,36$, we check directly that $N(m,5)<N(m,7)$ and that $N(m,5)<\frac{10m}{11}+2\leq\frac{(k-1)m}{k}+2<N(m,k)$ for all $k\geq11$, so these are not exceptions.

Now suppose $m$ is a multiple of $30$ but not a multiple of $7$. Let $k\neq7$ be coprime to $m$; we must then have $k\geq11$. Using Lemma~\ref{L:Lbound}, when $m\not\equiv1\pmod{7}$ and $m\geq55$, we have that $N(m,7)\leq\frac{6m+34}{7}\leq\frac{10m}{11}+2<N(m,k)$. Also by Lemma~\ref{L:Lbound}, when $m\equiv1\pmod{7}$ and $m\geq99$, we have that $N(m,7)=\frac{6m+50}{7}\leq\frac{10m}{11}+2<N(m,k)$. Therefore the only possible exception in this case is $m=30$. However one can check directly that $N(30,7)=30$ and $N(30,k)>30$ for all $k\geq11$, so this is not an exception.

Next suppose $m$ is a multiple of $210$ but not a multiple of $11$. Let $k\neq11$ be coprime to $m$; we must then have $k\geq13$. Using Lemma~\ref{L:Lbound}, when $m\not\equiv1\pmod{11}$ and $m\geq325$, we have that $N(m,11)\leq\frac{10m+72}{11}\leq\frac{12m}{13}+2<N(m,k)$. Also by Lemma~\ref{L:Lbound}, when $m\equiv1\pmod{7}$ and $m\geq650$, we have that $N(m,11)=\frac{10m+122}{11}\leq\frac{12m}{13}+2<N(m,k)$. Therefore the only possible exception in this case is $m=210$. Indeed, this is an exception, as one can check directly that $N(210,11)=202$, $N(210,13)=200$, and $N(210,k)>200$ for all $k\geq17$.

Finally, we consider all remaining cases at once. Let $k\geq 13$, and suppose $m$ is a multiple of all primes less than $k$ but not a multiple of $k$. Recall that the product of all primes less than $x$ is given by $e^{\theta(x)}$ where $\theta(x)=\sum_{p<x} \log(p)$ is Chebyshev's function.  It is well known that $\theta(x) \sim x$, and in particular it follows from~\cite{Cheb} that $\theta(x) > x-\frac{1}{\log(x)}$ for all $x \ge 41$.  Checking the remaining cases by hand, one sees that the product of all primes less than $k$ is greater than $\frac{k^3-3k+2}{2}$ for all $k\geq13$. Therefore $m\geq\frac{k^3-3k+2}{2}$. It follows that
\[
N(m,k)\leq \frac{(k-1)m+1}{k}+k\leq\frac{(k+1)m}{k+2}+2\leq\frac{(k+c-1)m}{k+c}+2<N(m,k+c)
\]
for all integers $c\geq2$. Therefore, for such $m$, the construction in Theorem~\ref{T:construction} is optimal for the smallest prime $k$ that does not divide $m$.
\end{proof}

We conclude by noting that, since the preceding theorem completely determines the positions with nonzero entries in Table~\ref{table:dist}, we can obtain Theorem~\ref{thm:criticalgp} as a corollary of Theorem~\ref{thm:critgporder}.

\section*{Acknowledgments}

This project began through the Research Experiences for Undergraduate Faculty (REUF) program which met at the Institute for Computational and Experimental Research in Mathematics (ICERM) on Brown University's campus during the summer of 2017. The authors would like to thank both the American Institute of Mathematics (AIM) and ICERM who supported our REUF through National Science Foundation (NSF) grants DMS-1620073 and DMS-1620080. The authors were further supported by an REUF continuation award during the summer of 2018. Without the support of these grants and program, this project would not have been possible. A.~Diaz-Lopez thanks the AMS and Simons Foundation for support under the AMS-Simons Travel Grant. J.~Louwsma was partially supported by a Niagara University Summer Research Award.

\providecommand{\bysame}{\leavevmode\hbox to3em{\hrulefill}\thinspace}
\providecommand{\MR}{\relax\ifhmode\unskip\space\fi MR }
\providecommand{\MRhref}[2]{%
  \href{http://www.ams.org/mathscinet-getitem?mr=#1}{#2}
}
\providecommand{\href}[2]{#2}

\end{document}